\documentclass[a4paper, 12pt,reqno]{amsart}
\usepackage[utf8]{inputenc}
\usepackage{blindtext}
\usepackage{setspace} 
\usepackage{a4wide}
\usepackage{amsthm}
\usepackage{amssymb}
\usepackage{longtable}
\usepackage{tikz-cd}
\usepackage{dynkin-diagrams}
\usepackage{amssymb,amsfonts}
\usepackage[all,arc]{xy}
\usepackage{enumerate}
\usepackage{mathrsfs}
\newtheorem{thm}{Theorem}[section]
\newtheorem{cor}[thm]{Corollary}
\newtheorem{prop}[thm]{Proposition}
\newtheorem{lem}[thm]{Lemma}

\theoremstyle{definition}

\theoremstyle{remark}
\newtheorem{rem}[thm]{Remark}

\makeatletter
\let\c@equation\c@thm
\makeatother
\numberwithin{equation}{section}

\begin{document}
	\title{Minimal parabolic subgroups and automorphism groups of Schubert varieties}
	
	\author[S. Senthamarai Kannan]{S. Senthamarai Kannan$^{1}$}
	\address{Chennai Mathematical Institute, Plot H1, SIPCOT IT Park, Siruseri, Kelambakkam,  603103, India.}
	\email{kannan@cmi.ac.in}
	\author[Pinakinath Saha]{Pinakinath Saha}
	\address{Tata Inst. of Fundamental Research, Homi Bhabha Road, Colaba, Mumbai 400005, India.}
	\email{psaha@math.tifr.res.in.}
	
	\begin{abstract} 	
		Let $G$ be a simple simply-laced algebraic group of adjoint type over the field $\mathbb{C}$ of complex numbers, $B$ be a Borel subgroup of $G$ containing a maximal torus $T$ of $G.$ In this article, we show that $\omega_\alpha$ is a minuscule fundamental weight  if and only if for any parabolic subgroup $Q$ containing $B$ properly, there is no Schubert variety $X_{Q}(w)$ in $G/Q$ such that the minimal parabolic subgroup $P_{\alpha}$ of $G$ is the  connected component, containing the identity automorphism of the group of all algebraic automorphisms of $X_{Q}(w).$	
	\end{abstract}

	\footnotetext[1]{Corresponding author: S. Senthamarai Kannan (Email: kannan@cmi.ac.in).}	
	
	
	
	\maketitle

	\section{Introduction}
	We recall that if $X$ is a projective variety over $\mathbb{C}$, the connected component, containing the identity automorphism of the group of all algebraic automorphisms of $X$ is an algebraic group (see \cite[Theorem 3.7, p.17]{MO}). On the other hand, M. Brion proved that every connected algebraic group $H$ over $\mathbb{C}$ is the connected component, containing identity automorphism of the group of all algebraic automorphisms of some normal projective variety $X$ (see \cite[Theorem 1]{Bri1}). Let $G$ be a simple algebraic group of adjoint type over $\mathbb{C}.$ Let $T$  be a maximal torus  of $G,$  and let $R$ be the set of roots with respect to $T.$ Let $R^{+}\subset R$ be a set of positive roots. Let $B^{+}$ be the Borel subgroup of $G$ containing $T,$ corresponding to $R^{+}$. Let $B$ be the Borel subgroup of $G$ opposite to $B^+$ determined by $T$. Let $W=N_{G}(T)/T$ denote the Weyl group of $G$ with respect to $T.$ 
	For $w \in W$, let $X(w):=\overline{BwB/B}$ denote the Schubert variety in $G/B$ corresponding to 
	$w$.
	In \cite{Dem aut}, M. Demazure studied the automorphism group of the homogeneous space $G/P,$  where $P$ is a parabolic subgroup of $G.$ The connected component, containing the identity automorphism of the group of all algebraic automorphisms of $G/P$ is $G,$ provided $(G, P)$ is not one of the following:
	\begin{enumerate}
		\item $G$ is of type $B_{n}$ and $P$ is the maximal parabolic subgroup corresponding to the simple root $\alpha_{n}.$
		
		\item $G$ is of type $C_{n}$ and $P$ is the maximal parabolic subgroup corresponding to the simple root $\alpha_{1}.$
		
		\item $G$ is of type $G_{2}$ and $P$ is the maximal parabolic subgroup corresponding to the simple root $\alpha_{1}.$
	\end{enumerate}
	
	The Lie algebra of $G$ may be identified with the Lie algebra of global vector fields $H^0(G/P, T_{G/P}).$ Let $Aut^{0}(X(w))$ denote the connected component, containing the identity automorphism of the group of all algebraic automorphisms of $X(w).$ 
	Let $\alpha_{0}$ denote the highest root of $G$ with respect to $T$ and $B^{+}.$ For the left action of $G$ on $G/B$, let $P_{w}$ denote the stabilizer of $X(w)$ in $G.$
	In \cite[p.772,Theorem 4.2(2)]{Kan}, the first named author proved that if $G$ is simply-laced and $X(w)$ is smooth, then we have $P_{w}=Aut^{0}(X(w))$ if and only if $w^{-1}(\alpha_{0})<0.$ Therefore, it is a natural question to ask whether given any parabolic subgroup $P$ of $G$ containing $B$ properly, is there a Schubert variety $X(w)$ in $G/B$ such that $P=Aut^{0}(X(w))$ ? If $P=B,$ there is no such Schubert variety in $G/B.$ In \cite{KP1}, authors gave an affirmative answer to this question. Also, authors gave some partial results for Schubert varieties in partial flag varieties of type $A_{n}.$ In this article, we study minimal parabolic subgroup $P_{i}$ for which there exists a Schubert variety $X(w_{i})$ in a partial flag variety such that $P_{i}=Aut^{0}(X(w_{i})).$ We prove the following.
	\begin{thm}\label{thm 1.1}
		Assume that $G$ is simply-laced. A fundamental weight $\omega_{\alpha}$ is  minuscule if and only if for any parabolic subgroup $Q$ containing $B$ properly, there is no Schubert variety $X_{Q}(w)$ in $G/Q$ such that $P_{\alpha}=Aut^0(X_{Q}(w))$ (see Theorem \ref{theorem 10.1}).
	\end{thm}
	The organization of the paper is as follows. In Section 2, we recall some preliminaries on algebraic groups and Lie algebras, a result from \cite{CP} and some results on cohomology of vector bundles on Schubert varieties (see \cite[p.271-272]{Dem}). Further, we recall a lemma on indecomposable modules from \cite{BKS} which will be used in computing the cohomology modules. We conclude this section by proving a generalization of the result \cite[Theorem 4.2, p.772]{Kan}. In Section 3, we prove some results on minuscule fundamental weights, co-minuscule simple roots. In Section 4, we prove some results on non minuscule fundamental weights. In Section 5, we prove that for any non minuscule fundamental weight $\omega_{i}$ in type $D,$ there exists a Schubert variety $X_{P_{i}}(w_{i})$ in $G/P_{i}$ such that $P_i$ is the connected component, containing the identity automorphism of the group of all algebraic automorphisms of $X_{P_i}(w_i)$ (for precise notation see section 2). In Section 6, we prove that for any non minuscule fundamental weight $\omega_{i}$ in type $E_6,$ there exists a Schubert variety $X_{P_{4}}(w_{i})$ in $G/P_{4}$ such that $P_i$ is the connected component, containing the identity automorphism of the group of all algebraic automorphisms of $X_{P_4}(w_i)$ (for precise notation see section 2). In Section 7, we prove that for any non minuscule fundamental weight $\omega_{i}$ in type $E_7,$ there exists a Schubert variety $X_{P_{3}}(w_{i})$ in $G/P_{3}$ such that $P_i$ is the connected component, containing the identity automorphism of the group of all algebraic automorphisms of $X_{P_3}(w_i)$ (for precise notation see section 2). In Section 8, we prove that for any fundamental weight $\omega_{i}$ in type $E_8,$ there exists a Schubert variety $X_{P_{7}}(w_{i})$ in $G/P_{7}$ such that $P_i$ is the connected component, containing the identity automorphism of the group of all algebraic automorphisms of $X_{P_7}(w_i)$ (for precise notation see section 2). In Section 9, we prove Theorem \ref{thm 1.1}.
	\section{Notation and Preliminaries}
	In this section, we set up some notation and preliminaries. We refer to \cite{BK}, \cite{Hum1}, \cite{Hum2}, \cite{Jan} for preliminaries in algebraic groups and Lie algebras.

	Let $G,B,T, R, R^{+},$ and $W,$ be as in the introduction.    
	Let $S = \{\alpha_1,\ldots,\alpha_n\}$ denote the set of simple roots in $R^{+}.$ Every $\beta \in R$ can be expressed uniquely as  $\sum\limits_{i=1}^{n}k_{i}\alpha_{i}$ with integral coefficients $k_{i}$ all non-negative sign or non-positive. This allows us to define the {\bf height} of a root (relative to $S$) by ht$(\beta)=\sum\limits_{i=1}^{n}k_{i}.$ For $\beta=\sum\limits_{i=1}^{n}k_{i}\alpha_{i} \in R,$ we define {\bf support} of $\beta$ to be the set $\{\alpha_{i}: k_{i}\neq 0 \}.$
	The simple reflection in  $W$ corresponding to $\alpha_i$ is denoted
	by $s_{i}$. Then $(W, S)$ is a Coxeter group (see \cite[Theorem 29.4, p.180]{Hum2}). There is a natural length function $\ell$ defined on $W.$ Let $\mathfrak{g}$ be the Lie algebra of $G$. 
	Let $\mathfrak{h}\subset \mathfrak{g}$ be the Lie algebra of $T$ and  $\mathfrak{b}\subset \mathfrak{g}$ be the Lie algebra of $B$. Let $X(T)$ denote the group of all characters of $T$. 
	We have $X(T)\otimes\mathbb{R}=Hom_{\mathbb{R}}(\mathfrak{h}_{\mathbb{R}}, \mathbb{R})$, the dual of the real form of $\mathfrak{h}$. The positive definite 
	$W$-invariant form on $Hom_{\mathbb{R}}(\mathfrak{h}_{\mathbb{R}}, \mathbb{R})$ 
	induced by the Killing form of $\mathfrak{g}$ is denoted by $(~,~)$. 
	We use the notation $\left< ~,~ \right>$ to
	denote $\langle \mu, \alpha \rangle  = \frac{2(\mu,
		\alpha)}{(\alpha,\alpha)}$,  for every  $\mu\in X(T)\otimes \mathbb{R}$ and $\alpha\in R$. For a subset $J$ of $S,$ we denote by $W_{J}$ the subgroup of $W$ generated by $\{s_{\alpha}:\alpha \in J\}$. Let $W^{J}:=\{w\in W: w(\alpha)\in R^{+}~ for ~ all ~ \alpha \in J\}.$ We denote by $w_{0}$ the longest element of $W.$ Note that for $J\subseteq S,$ and $w\in W,$ there are unique elements $w_{J}\in W_{J}$ and $w^{J}$ in $W^{J}$ such that $w=w^{J}w_{J}.$ We denote by $w_{0,J}$ the longest element of $W_{J}.$ Note that $w_{0,J}=(w_{0})_{J}.$ Further, we denote by $w_{0}^{J}$ the minimal representative in $W^{J}$ of $w_{0}.$ For $w\in W,$ let $R^{+}(w):=\{\beta \in R^{+}: w(\beta)<0\}.$ For a root $\alpha,$ let $U_{\alpha}$ be the root subgroup of $G$ corresponding $\alpha.$ 
	
	For each $w\in W_{J},$ choose a representative element $n_{w}\in N_{G}(T).$ Let $N_{J}:=\{n_{w}: w\in W_{J}\}.$  Let $P_{J}:=BN_{J}B.$
	For a simple root $\alpha_{i}$, we denote by $P_{i}$ the minimal parabolic subgroup $P_{\alpha_{i}}$ of $G.$  Let  $\{\omega_{i}:1\le i\le n\}$ be the set of fundamental dominant weights corresponding to $\{\alpha_{i}: 1\le i\le n\}.$ For $1\le i \le n,$ let $h(\alpha_{i})\in \mathfrak{h}$ be the fundamental co-weight corresponding to $\alpha_{i}.$ That is ; $\alpha_{i}(h(\alpha_{j})) = \delta_{ij},$ where $\delta_{ij}$ is Kronecker delta.

	We recall the following definition and facts (see \cite[p. 119-120]{BL}):
	
	A fundamental weight $\omega$ is said to be  {\bf minuscule} if $\omega$ satisfies $\langle\omega, \beta \rangle \le 1$ for all $\beta \in R^{+}.$ The following is the complete list of minuscule weights in the simply-laced root systems.
	
	%\vspace{.1cm}
	\begin{center}
		\begin{tabular}{ |p{1.3cm}|p{3.3cm}|p{3.5cm}| }
			\hline
			\multicolumn{3}{|c|}{ Table 1: Minuscule weight in simply-laced root system} \\
			\hline
			no.&Root System& Minuscule weight \\
			\hline
			$1.$ &	$A_{n}$  & $\omega_{1}, \omega_{2}, ..., \omega_{n}$      \\
			%\vspace{.1cm}&\vspace{.1cm}& \vspace{.1cm}\\
			
			2. & $D_{n} $  & $\omega_{1}, \omega_{n-1}, \omega_{n}$       \\
			3.&	$E_{6}$  & $\omega_{1},\omega_{6}$       \\
			4.&	$E_{7}$  &    $\omega_{7}$    \\
			5.& $E_{8}$  & none      \\
			\hline
		\end{tabular}
	\end{center}
	A simple root $\alpha$ is said to be  {\bf co-minuscule} if $\alpha$ occurs with coefficient $1$ in the expression of the highest root $\alpha_{0}.$ A fundamental weight $\omega_{\alpha}$ associated to a simple root $\alpha$ is said to be {\bf co-minuscule} if $\alpha$ is {\bf co-minuscule}.

	Since $G$ is simply-laced, the list of minuscule weights is also the list of co-minuscule weights (see Lemma \ref{lem 3.2}). 
	
	We recall the following Proposition from \cite[Proposition 7.1, page 342-343]{CP}.
	
	Let $\alpha_{0}=\sum\limits_{i=1}^{n}c_{i}\alpha_{i},$ and $\check{\alpha_{0}}=\sum\limits_{i=1}^{n}\check{c_{i}}\check{\alpha_{i}}.$ We have $\check{\alpha_{0}}=\frac{2\alpha_{0}}{(\alpha_{0}, \alpha_{0})}=\frac{2}{(\alpha_{0}, \alpha_{0})}\sum\limits_{i=1}^{n}c_{i}\frac{(\alpha_{i}, \alpha_{i})}{2}\check{\alpha_{i}},$ hence $\check{c_{i}}=\frac{(\alpha_{i}, \alpha_{i})}{(\alpha_{0}, \alpha_{0})}c_{i}.$ If $G$ is simply-laced, we have $c_{i}=\check{c_{i}}.$ The dual Coxeter number of $\mathfrak{g}$ is 
	$$g=1+\sum\limits_{i=1}^{n}\check{c_{i}}.$$ 
	\begin{prop}\label{Prop 2.1}
		Let $\alpha$ be any long root. Then we have 
		\begin{itemize}
			\item[(1)] There is a unique element $u_{\alpha}$ in $W$ of minimal length such that $u_{\alpha}^{-1}(\alpha_{0})=\alpha.$
		\end{itemize}
		
		\begin{itemize}
			\item[(2)] If $\alpha$ is in $S,$ then $\ell(u_{\alpha})=g-2.$
		\end{itemize}
	\end{prop}
	\begin{proof}
		See \cite[Proposition 7.1, page 342-343]{CP}.
	\end{proof}
	\begin{cor}\label{corollary 2.1}
		If $G$ is simply-laced and $\alpha$ is in $S,$ then we have $\ell(v_{\alpha})=ht(\alpha_{0})$ where $v_{\alpha}=u_{\alpha}s_{\alpha}.$
	\end{cor}
	\begin{proof}
		By Proposition \ref{Prop 2.1}, we have $\ell(u_{\alpha})=g-2.$ Further, we note that $\ell(v_{\alpha})=\ell(u_{\alpha})+1=g-1=\sum\limits_{i=1}^{n}\check{c_{i}}.$ If $G$ is simply-laced, we have $(\alpha_{0}, \alpha_{0})=(\alpha_{i}, \alpha_{i})$ for all $1\le i\le n.$ Therefore, $\ell(u_{\alpha})=\sum\limits_{i=1}^{n}c_{i}$=ht$(\alpha_{0}),$ as $\check{c_{i}}=\frac{(\alpha_{i}, \alpha_{i})}{(\alpha_{0}, \alpha_{0})}c_{i}=c_{i}.$ 
	\end{proof}
	
	Now we discuss some preliminaries on the cohomology of vector bundles on Schubert varieties associated to the rational $B$-modules.
	
	Let $V$ be a rational $B$-module. Let $\phi:B\longrightarrow GL(V)$ be the corresponding homomorphism of algebraic groups. The total space of the vector bundle  $\mathcal{L}(V)$  on $G/B$ is defined by the set of equivalence classes 
	$\mathcal{L}(V)= G \times_{B} V$ corresponding to the following equivalence relation on $G\times V$:
	\begin{center}
		$(g,v)\sim (gb,\phi(b^{-1})\cdot v)$ for $g\in G, b\in B, v\in V.$ 
	\end{center}
	
	We denote by the restriction of $\mathcal{L}(V)$ to $X(w)$ also by $\mathcal{L}(V)$. We denote the cohomology modules $H^i(X(w), \mathcal{L}(V))$ by $H^i(w, V)$ ($i\in \mathbb{Z}_{\ge 0}$). If $V=\mathbb{C}_{\lambda}$ is a one dimensional representation $\lambda: B\longrightarrow \mathbb{C}^{\times}$ of $B,$ then we denote $H^i(w, V)$ by $H^i(w, \lambda).$
	
	Let $L_{\alpha}$ denote the Levi subgroup of $P_{\alpha}$
	containing $T$. Note that $L_{\alpha}$ is the product of $T$ and the homomorphic image 
	$G_{\alpha}$ of $SL(2, \mathbb{C})$ via a homomorphism $\psi:SL(2, \mathbb{C})\longrightarrow L_{\alpha}$ (see  [7, II, 1.3]). We denote the intersection of $L_{\alpha}$ and $B$ by $B_{\alpha}$.  
	We note that the morphism $L_{\alpha}/B_{\alpha}\hookrightarrow P_{\alpha}/B$  induced by the inclusion $L_{\alpha}\hookrightarrow P_{\alpha}$ is an isomorphism. Therefore,  to compute the cohomology modules $H^{i}(P_{\alpha}/B, \mathcal{L}(V))$ ($0\leq i \leq 1$) for any $B$-module 
	$V,$ we treat $V$ as a $B_{\alpha}$-module  and we compute 
	$H^{i}(L_{\alpha}/B_{\alpha}, \mathcal{L}(V))$.

	For $\lambda \in X(T)$ and $w\in W,$ we define the dot action by the rule $w\cdot \lambda\,=\, w(\lambda + \rho)-\rho,$ where $\rho$ is the half sum of positive roots of $G.$ 
	
	The following lemma is due to Demazure (see \cite[p.271-272]{Dem}). We use this lemma to compute cohomology modules.
	
	\begin{lem}\label{lem 2.1}
		Let $w=\tau s_{\alpha},$ $\ell(w)=\ell(\tau) + 1,$ and $\lambda$ be a character of $B.$ Then we have 
		\begin{enumerate}
			
			\item [(1)] If $\langle \lambda , \alpha \rangle\ge0,$ then $H^{j}(w, \lambda)=H^{j}(\tau, H^0(s_{\alpha}, \lambda))$ for all $j\ge 0.$
			
			\item[(2)] If $\langle \lambda , \alpha \rangle\ge 0,$ then $H^j(w, \lambda)=H^{j+1}(w, s_{\alpha}\cdot \lambda)$ for all $j\ge0.$ 
			
			\item[(3)] If $\langle \lambda, \alpha\rangle\le -2 ,$ then $H^{j+1}(w, \lambda)=H^j(w, s_{\alpha}\cdot \lambda)$ for all $j\ge 0.$
			
			\item[(4)] If $\langle \lambda , \alpha \rangle=-1,$ then $H^j(w, \lambda)$ vanishes for every $j\ge0.$
		\end{enumerate}
	\end{lem}
	
	Let $\pi:\hat{G}\longrightarrow G$ be the simply connected covering of $G$.  
	Let  $\hat{L_{\alpha}}$  (respectively, $\hat{B_{\alpha}}$)  be the inverse image 
	of $L_{\alpha}$ (respectively, of $B_{\alpha}$) in $\hat{G}$. Note that $\hat{L_{\alpha}}/\hat{B_{\alpha}}$ is isomorphic to $L_{\alpha}/B_{\alpha}$. We make use of this isomorphism to use the same notation for the vector bundle on $L_{\alpha}/B_{\alpha}$ associated to a $\hat{B_{\alpha}}$-module. Let $V$ be an irreducible  $\hat{L_{\alpha}}$-module and $\lambda$ be a character of $\hat{B_{\alpha}}.$
	
	Then, we have
	
	\begin{lem}\label{lem2.3}
		\begin{enumerate}
			
			\item If $\langle \lambda , \alpha \rangle \geq 0$, then the $\hat{L_{\alpha}}$-module
			$H^{0}(L_{\alpha}/B_{\alpha} , V\otimes \mathbb{C}_{\lambda})$ 
			is isomorphic to the tensor product of  $ ~ V$ and 
			$H^{0}(L_{\alpha}/B_{\alpha} , \mathbb{C}_{\lambda})$. Further, we have  
			$H^{j}(L_{\alpha}/B_{\alpha} , V\otimes \mathbb{C}_{\lambda}) =0$ 
			for every $j\geq 1$.
			\item If $\langle \lambda , \alpha \rangle  \leq -2$, then we have  
			$H^{0}(L_{\alpha}/B_{\alpha} , V \otimes \mathbb{C}_{\lambda})=0.$ 
			Further, the $\hat{L_{\alpha}}$-module  $H^{1}(L_{\alpha}/B_{\alpha} , V \otimes \mathbb{C}_{\lambda})$ is isomorphic to the tensor product of  $~V$ and $H^{0}(L_{\alpha}/B_{\alpha} , 
			\mathbb{C}_{s_{\alpha}\cdot\lambda})$.
			\item If $\langle \lambda , \alpha \rangle = -1$, then 
			$H^{j}( L_{\alpha}/B_{\alpha} , V \otimes \mathbb{C}_{\lambda}) =0$ 
			for every $j\geq 0$.
		\end{enumerate}
	\end{lem}
	
	\begin{proof}
		
		By \cite[I, Proposition 4.8, p.53]{Jan} and \cite[I, Proposition 5.12, p.77]{Jan} for $j\ge0$,
		we have the following isomorphism as 
		$\hat{L_{\alpha}}$-modules:
		$$H^j(L_{\alpha}/B_{\alpha}, V \otimes \mathbb C_{\lambda})\simeq V \otimes
		H^j(L_{\alpha}/B_{\alpha}, \mathbb C_{\lambda}).$$ 
		Now the proof of the lemma follows from Lemma \ref{lem 2.1} by taking $w=s_{\alpha}$ 
		and the fact that $L_{\alpha}/B_{\alpha} \simeq P_{\alpha}/B$.
	\end{proof}
	
	We now state the following lemma on indecomposable 
	$\hat{B_{\alpha}}$ (respectively,  $B_{\alpha}$) modules which will be used in computing 
	the cohomology modules (see \cite [Corollary 9.1, p.130]{BKS}).
	
	\begin{lem}\label{lem2.4}
		\begin{enumerate}
			\item
			Any finite dimensional indecomposable $\hat{B_{\alpha}}$-module $V$ is isomorphic to 
			$V^{\prime}\otimes \mathbb{C}_{\lambda}$ for some irreducible representation
			$V^{\prime}$ of $\hat{L_{\alpha}}$, and some character $\lambda$ of $\hat{B_{\alpha}}$.
			
			\item
			Any finite dimensional indecomposable $B_{\alpha}$-module $V$ is isomorphic to 
			$V^{\prime}\otimes \mathbb{C}_{\lambda}$ for some irreducible representation
			$V^{\prime}$ of $\hat{L_{\alpha}}$, and some character $\lambda$ of $\hat{B_{\alpha}}$.
		\end{enumerate}
	\end{lem}
	\begin{proof}
		Proof of part (1) follows from \cite [Corollary 9.1, p.130]{BKS}.
		
		Proof of part (2) follows from the fact that every $B_{\alpha}$-module can be viewed  as a 
		$\hat{B_{\alpha}}$-module via the natural homomorphism. 
	\end{proof} 
	
	We conclude this section by proving a generalization of \cite[Theorem 4.2, p.772]{Kan} which describes the automorphism group of a smooth Schubert variety in the simply laced case. We use this result frequently in our article.
	
	Assume that $G$ is simply-laced. Let $P=P_{I}$ be the standard parabolic subgroup of $G$ corresponding to a subset $I\subseteq S.$ Let $w\in W^I,$ and $X_{P}(w):=\overline{BwP/P}$ be the Schubert variety in $G/P$ corresponding to $w.$ For the left action of $G$ on $G/P,$ let $P_{w}$ denote the stabilizer of $X_{P}(w)$ in $G.$
	\begin{thm}\label{thm1}
		Then we have
		\item$(1)$ The homomorphism $\varphi_{w}:P_{w}
		\longrightarrow Aut^{0}(X_{P}(w))$ induced by the action of $P_{w}$ on $X_{P}(w)$ is surjective.
		
		\item $(2)$ $\varphi_{w}: P_{w}\longrightarrow Aut^{0}(X_{P}(w))$ is an isomorphism if and only if $w^{-1}(\alpha_{0})$ is a negative root.
	\end{thm}
	\begin{proof}
		
		Proof of (1):
Let $T_{X_{P}(w)}$ (respectively, $T_{G/P}$) be the tangent sheaf of $X_{P}(w)$ (respectively, of $G/P$). We denote the restriction of $T_{G/P}$ to $X_{P}(w)$ also by $T_{G/P}.$ Note that $T_{G/P}=\mathcal{L}(\mathfrak{g/p}),$ and $T_{X_{P}(w)}$ is a subsheaf of $T_{G/P}.$
		
	Let $\pi: G/B\longrightarrow G/P$ be the natural map. Then the restriction of $\pi$ to $X(w),$ i.e., $\pi: X(w)\longrightarrow X_{P}(w)$ is a birational surjective morphism. Further, by \cite[Theorem 3.3.4(a), p.96]{BK}, $\pi^*:H^j(X_{P}(w), \mathcal{L}_{P}(V))\longrightarrow H^j(X(w), \pi^*\mathcal{L}_{P}(V))$ ($j\ge 0$) is an isomorphism for any $P$-module $V.$ Thus we denote the cohomology modules $H^j(X_{P}(w), \mathcal{L}_{P}(V))(\simeq H^j(X(w), \pi^*\mathcal{L}_{P}(V)))$ by simply $H^j(w, V).$
		
		Now we consider the short exact sequence 
		\begin{center}
			$0\longrightarrow \mathfrak{p/b}\longrightarrow \mathfrak{g/b}\longrightarrow\mathfrak{g/p}\longrightarrow0$
		\end{center}
		of $B$-modules.
		
Let $R^{+}(P)$ be the set of positive roots of the Levi factor of $P$ containing $T.$ Observe that $\mathfrak{p/b}$ has a filtration of $B$-submodules such that the successive quotients are of the form $\mathbb{C}_{\beta},$ where $\beta\in R^{+}(P).$ Further, by \cite[Corollary 3.6, p.771]{Kan} we have $H^j(w,\beta)=0$ ($j\ge 1$) for any positive root $\beta.$ 
		
	Hence, by using the long exact sequence associated to the above short exact sequence, we see that $H^0(w,\mathfrak{g/b})\longrightarrow H^0(w,\mathfrak{g/p})$ is surjective.

		On the other hand, by using \cite[Lemma 3.5, p.770]{Kan}, the restriction map $$H^0(w_0, \mathfrak{g/b})\longrightarrow H^0(w,\mathfrak{g/b})$$ is surjective. Hence, the restriction map $H^0(w_{0}^I, \mathfrak{g/b})\longrightarrow H^0(w, \mathfrak{g/b})$ is surjective.
		
		Further, we have the following commutative diagram:
		
		\[ \begin{tikzcd}
			H^0(w_{0}^{I},\mathfrak{g/b})\arrow[twoheadrightarrow]{r} \arrow[swap, twoheadrightarrow]{d}{} & H^0(w_{0}^{I},\mathfrak{g/p})\arrow{d}{r} \\%
			H^0(w,\mathfrak{g/b}) \arrow[twoheadrightarrow]{r}{}& H^0(w, \mathfrak{g/p})
		\end{tikzcd}
		\]

		Hence, the natural restriction map $r:H^0(w_{0}^{I},\mathfrak{g/p})\longrightarrow H^0(w, \mathfrak{g/p})$ is surjective.
		
		Since $P_w$ is a parabolic subgroup of $G$ containing $B,$ the Lie algebra $\mathfrak{p}_w$ is a Lie subalgebra of $\mathfrak{g}= H^0(w_0^{I}, \mathfrak{g/p})$ containing $\mathfrak{b}.$ Further, since $X_{P}(w)$ is a
		$P_w$-stable subvariety for the left action of $P_w$ on $G/P,$ we have the following commutative
		diagram of $B$-modules:

		\[ \begin{tikzcd}
			\mathfrak{p}_{w} \arrow{r}{d\varphi_w} \arrow[swap, hook]{d}{} & H^0(X_{P}(w), T_{X_{P}(w)})\arrow[hook]{d}{} \\%
			H^0(w_{0}^{I}, \mathfrak{g/p}) \arrow[twoheadrightarrow]{r}{r}& H^0(w, \mathfrak{g/p})
		\end{tikzcd}
		\]

		Now, let $\mathfrak{q} = r^{-1}(H^0(X_{P}(w), T_{X_{P}(w)})).$ Note that since $\mathfrak{q}$ is a $B$-submodule of $\mathfrak{g}$ containing
		$\mathfrak{p}_w,$ $\mathfrak{q}$ is a parabolic subalgebra of $\mathfrak{g}$ containing $\mathfrak{p}_w.$ We denote the restriction of $r$ to $\mathfrak{q}$ also
		by $r.$ We now show that $\mathfrak{p}_{w}= \mathfrak{q}.$ Since $\mathfrak{g}= H^0(w_0^{I}, \mathfrak{g/p}),$ every element $x\in \mathfrak{q}\subseteq \mathfrak{g}$ is a tangent vector field on $G/P.$ Further, by the definition of $\mathfrak{q}$; the ideal sheaf of $X_{P}(w)$ is $x$ stable for every $x\in \mathfrak{q}.$ Therefore, $\mathfrak{q}$ is contained in the Lie algebra $\mathfrak{p}_w$ of the stabiliser $P_w$ of $X_P(w)$ in $G.$ Thus, we have $\mathfrak{p}_w =\mathfrak{q}.$
		
		Clearly, $r:\mathfrak{p}_{w}=\mathfrak{q} \longrightarrow H^0(X_{P}(w), T_{X_{P}(w)})$ is a homomorphism of Lie algebras. Therefore $d\varphi_{w}:\mathfrak{p}_w\longrightarrow H^0(X_{P}(w), T_{X_{P}(w)})$ is surjective. Let $A=\varphi_w(P_{w})\subseteq Aut^0(X_P(w)).$ Then we have $d\varphi_w(\mathfrak{p}_{w})=Lie(A)\subseteq H^0(X_{P}(w), T_{X_{P}(w)}).$ Since $d\varphi_{w}: \mathfrak{p}_{w}\longrightarrow H^0(X_{P}(w), T_{X_{P}(w)})$ is surjective, we have $Lie(A)=H^0(X_{P}(w), T_{X_{P}(w)}).$ Therefore, we have $A=Aut^0(X_{P}(w))$ and hence, $\varphi_w: P_{w}\longrightarrow Aut^0(X_{P}(w))$ is surjective.

		Proof of (2):
		Assume that $w^{-1}(\alpha_{0})$ is a negative root.
		Then we claim that $H^0(w, \mathfrak{p})=0.$

	Note that $\mathfrak{p}$ has a filtration of $B$-submodules such that the successive quotients are of the form $\mathbb{C}_{0}$ or $\mathbb{C}_{\beta}$ for some $\beta\in R^{+}(P)\cup R^{-}.$  So, every weight of $H^0(w,\mathfrak{p})$ is a weight of $H^0(w,0)$ or $H^0(w,\beta)$ for some $\beta\in R^{+}(P)\cup R^{-}.$ Now, if $-\alpha_{0}$ is a weight of $H^0(w,\mathfrak{p}),$ then $-\alpha_{0}$ is a weight of $H^0(w,\beta)$ for some $\beta\in R^{+}(P)\cup R^{-}.$ Further, every weight $\mu$ of $H^0(w,\beta)$ satisfies $\mu \ge -\alpha_{0}.$ Thus, we have $w(\beta)=-\alpha_{0}.$ Since $w^{-1}(\alpha_{0})<0,$ it follows that $\beta \notin R^{-}.$  So, we have $\beta\in R^{+}(P).$
		
Now we show that $\beta\in R^{+}\setminus R^{+}(P).$  Since $w\in W^I,$ we have $R^{+}(w)\cap R^{+}(w_{0, I})=\emptyset.$ Now, since $\beta\in R^{+}(w),$ we have $\beta \notin R^{+}(w_{0,I}),$ i.e., $\beta\notin R^{+}(P),$ which is a contradiction to the fact that $\beta\in R^{+}(P).$ Hence, $-\alpha_{0}$ is not a weight of $H^0(w,\mathfrak{p}).$ 
		
		Now since $H^0(w, \mathfrak{p})$ is a $B$-submodule of $H^0(w,\mathfrak{g})=\mathfrak{g}$ such that $H^0(w,\mathfrak{p})_{-\alpha_{0}}$ is zero, and  $\mathfrak{g}_{-\alpha_{0}}$ is the unique $B$-stable line in $\mathfrak{g},$ it follows that $H^0(w, \mathfrak{p})=0.$ Therefore, the map $H^0(w,\mathfrak{g})\longrightarrow H^0(w, \mathfrak{g/p})$ induced by the map $\mathfrak{g}\longrightarrow \mathfrak{g/p}$ is injective. Further, by \cite[Lemma 3.4, p.770]{Kan}, the map $H^0(w,\mathfrak{g})\longrightarrow H^0(w,\mathfrak{g/p})$ is surjective. Thus, the map $H^0(w,\mathfrak{g})\longrightarrow H^0(w,\mathfrak{g/p})$ is an isomorphism. Hence, $H^0(w,\mathfrak{g/p})$ is isomorphic to $\mathfrak{g}$ as $B$-module. Therefore, $H^0(w, \mathfrak{g/p})$ has a unique $B$-stable line, namely $\mathfrak{g}_{-\alpha_{0}}.$ Note that since $w^{-1}(\alpha_{0})< 0,$ we have $w\neq id.$ Therefore, the action of $P_w$ on $X_{P}(w)$ is non trivial. Hence, the homomorphism $d\varphi_{w}: \mathfrak{p}_w\longrightarrow H^0(X_{P}(w),T_{X_{P}(w)})$ of $B$-modules is non-zero. Therefore, the $B$-stable line $H^0(w, \mathfrak{g/p})_{-\alpha_{0}}$ is in the image $d\varphi_w(\mathfrak{p}_w)\subset H^0(X_{P}(w), T_{X_{P}(w)})\subset H^0(w, \mathfrak{g/p}).$
		Hence, we have $\mathfrak{g}_{-\alpha_{0}}\cap \ker(d\varphi_w)= 0.$ Thus, $d\varphi_w:\mathfrak{p}_w \longrightarrow H^0(X_{P}(w),T_{X_{P}(w)})$ is injective.
		Since the base field is $\mathbb{C},$ it follows that $\varphi_w$ is separable. Hence, the kernel of $d\varphi_w$ is the Lie algebra of the kernel of $\varphi_{w}.$ Therefore, $\varphi_{w} : P_{w}\longrightarrow Aut^0(X_{P}(w))$ is injective. Now, $\varphi_{w}$ is an isomorphism follows from $(1).$

		Conversely, if $\varphi_w: P_w\longrightarrow Aut^0(X_{P}(w))$ is an isomorphism, then, the induced homomorphism
		$d\varphi_w: \mathfrak{p}_w\longrightarrow H^0(X_{P}(w), T_{X_{P}(w)})\subseteq H^0(w, \mathfrak{g/p})$ is injective. In particular, the $-\alpha_{0}$-weight
		space $H^0(w, \mathfrak{g/p})_{-\alpha_{0}}$ is non-zero.
		
We now show that $w^{-1}(\alpha_{0})<0.$ For this, we first note that the $B$-module $\mathfrak{g/p}$ has a composition series of $B$-modules with each successive simple quotient is isomorphic to $\mathbb{C}_{\alpha},$ where $\alpha\in R^{+}\setminus R^{+}(P).$ Applying \cite[Corollary 3.6, p.771]{Kan} we see that the
		$B$-module $H^0(w,\mathfrak{g/p})$ has a filtration of $B$-submodules with each successive quotient is isomorphic to $H^0(w,\alpha)$ for some $\alpha \in R^{+}\setminus R^{+}(P).$ Now, since $H^0(w, \mathfrak{g/p})_{-\alpha_{0}}\neq0,$   $H^0(w,\alpha)_{-\alpha_{0}}\neq0$ for some $\alpha\in R^{+}\setminus R^{+}(P).$ On the other
		hand, there is a $v\in W$ such that $v(\alpha_{0}) =\alpha.$ Without loss of generality, we may assume that $v$ is of minimum length among such elements. It follows from the Demazure character formula that if $H^0(wv,\alpha_{0})_{\mu}\neq 0,$ then $\mu$ is in the convex hull of the set $\{x(\alpha_{0}): x\le wv\}$ ( see \cite[Theorem 3.3.8, p.97, equation (3)]{BK} and \cite[Proposition 14.18(b),p.379]{Jan}). Note that the convention for the signature of the weights in the Demazure character formula in this paper is the same as the one in \cite{Jan}. Further, using the above arguments, we
		see that $\mathbb{C}_{\alpha}$ is a $B$-submodule of $H^0(v, \alpha_{0}).$ Therefore, $H^0(w, \alpha)$ is a $B$-submodule of $H^0(w,H^0(v, \alpha_{0}))=H^0(wv,\alpha_{0}).$ Hence, every weight $\mu$ of $H^0(w,\alpha)$ satisfies $\mu\ge  w(\alpha).$ Clearly, $w(\alpha)\ge -\alpha_{0}.$ Therefore, since $H^0(w, \alpha)_{-\alpha_{0}}\neq 0,$ we have $w(\alpha)=-\alpha_{0}.$ Thus, $w^{-1}(\alpha_{0})<0.$

	\end{proof}

	\section{Preliminaries on Minuscule fundamental weight}
	Now on wards we assume that $G$ is simply-laced. Note that $-w_{0}$ is an automorphism of the root system $R.$  Let $\sigma$ be the Dynkin diagram automorphism induced by $-w_{0}$ i.e. $\alpha_{i}\mapsto \alpha_{\sigma(i)}=-w_{0}(\alpha_{i})$ for $1\le i\le n.$

	\begin{lem}\label{lem 3.2} Let $1\le r\le n.$ Then $\alpha_{r}$ is co-minuscule  root if and only if $\omega_{r}$ is minuscule. 
	\end{lem}
	
	\begin{proof}
		Let $\alpha_{0}=\sum\limits_{i=1}^{n}c_{i}\alpha_{i}.$
		Suppose that the coefficient of $\alpha_{r}$ in $\alpha_{0}$ is $1.$ Then we have $\langle\omega_{r}, \alpha_{0} \rangle=1,$ as $G$ is simply laced.  Since $G$ is simply laced, we have $\langle \omega_{r}, \alpha_{0}\rangle\ge \langle \omega_{r}, \beta\rangle$ for all $\beta \in R^{+}.$ Hence $\omega_{r}$ is  minuscule.

		Conversely, we assume that $\omega_{r}$ is  minuscule. Since $G$ is simply laced, we have $\langle\omega_{r}, \alpha_{0} \rangle=c_{r}.$ Thus we have $c_{r}=1,$ as $\omega_{r}$ is  minuscule.
	\end{proof}
	
	\begin{lem}\label{lem 2.3}
		$\omega_{r}$ is  minuscule  if and only if $w_{0,S \setminus \{\alpha_{r}\}}(\alpha_{r})=\alpha_{0}.$ 
	\end{lem}
	\begin{proof}
		Assume that $\omega_{r}$ is minuscule.  
		Note that $-\alpha_{r}$ is $L_{S\setminus\{\alpha_{r}\}}$ dominant. Therefore, $w_{0, S\setminus\{\alpha_{r}\}}(-\alpha_{r})$ is $L_{S\setminus\{\alpha_{r}\}}$ negative dominant. Further, the coefficient of $\alpha_{r}$ in $w_{0,S\setminus\{\alpha_{r}\}}(-\alpha_{r})$ is $-1.$ 
		
		On the other hand, if $\langle w_{0, S\setminus\{\alpha_{r}\} }(-\alpha_{r}), \alpha_{r}\rangle\ge 1$ the coefficient of $\alpha_{r}$ in $s_{r}(w_{0,S\setminus\{\alpha_{r}\}}(-\alpha_{r}))$ is $\le- 2.$ Since $-\alpha_{0}\le s_{r}w_{0, S\setminus \{\alpha_{r}\}}(-\alpha_{r}),$ the coefficient of $\alpha_{r}$ in $-\alpha_{0}$ is $\le -2.$ This is a contradiction to Lemma \ref{lem 3.2}. Hence $w_{0, S\setminus\{\alpha_{r}\}}(-\alpha_{r})$ is negative dominant. Thus we have  $w_{0, S\setminus\{\alpha_{r}\}}(\alpha_{r})=\alpha_{0}.$
		
		Conversely, suppose that $w_{0,S\setminus\{\alpha_{r}\}}(\alpha_{r})=\alpha_{0}.$ Let $\alpha_{0}=\sum\limits_{i=1}^{n}c_{i}\alpha_{i}.$ Since $G$ is simply-laced, we have $\langle\omega_{r}, \alpha_{0}\rangle=c_{r}.$ 
		Further, since $w_{0, S\setminus\{\alpha_{r}\}}(\alpha_{r})=\alpha_{0},$ we have  $\langle \omega_{r},\alpha_{0}\rangle=\langle  \omega_{r}, w_{0, S\setminus\{\alpha_{r}\}}(\alpha_{r})\rangle.$ As $\langle\cdot, \cdot \rangle $ is $W$-invariant, we have $\langle \omega_{r}, w_{0, S\setminus\{\alpha_{r}\}}(\alpha_{r})\rangle=\langle w_{0, S\setminus\{\alpha_{r}\}}(\omega_{r}), \alpha_{r}\rangle.$ Since $w_{0,S\setminus\{\alpha_{r}\}}(\omega_{r})=\omega_{r},$ we have $\langle w_{0,S\setminus\{\alpha_{r}\}}(\omega_{r}),\alpha_{r})\rangle=
		\langle \omega_{r},\alpha_{r}\rangle=1.$ Thus we have $c_{r}=1.$ Therefore, by Lemma \ref{lem 3.2}, $\omega_{r}$ is minuscule.
	\end{proof}
	
	\begin{lem}\label{lem 2.4} Assume that $\omega_{r}$ is minuscule.
		Let $v\in W^{S\setminus\{\alpha_{r}\}}.$ Then we have $v={w_{0}^{S\setminus\{\alpha_{r}\}}}$ if and only if $v(\alpha_{0})<0.$ 
	\end{lem}
	\begin{proof}
		Assume that $v=w_{0}^{S\setminus\{\alpha_{r}\}}.$ Since  $w_{0}=w_{0}^{S\setminus\{\alpha_{r}\}}w_{0,S\setminus\{\alpha_{r}\}},$ we have $w_{0}^{S\setminus\{\alpha_{r}\}}=w_{0}w_{0,S\setminus\{\alpha_{r}\}}.$ Therefore, by Lemma \ref{lem 2.3} we have $w_{0}^{S\setminus\{\alpha_{r}\}}(\alpha_{0})=w_{0}(\alpha_{r})=-\alpha_{\sigma(r)}.$ 
		
		Conversely, Assume that $v(\alpha_{0})<0.$ By Lemma \ref{lem 2.3}, we have $s_{\sigma(r)}{w_{0}}^{S\setminus\{\alpha_{r}\}}(\alpha_{0})=\alpha_{\sigma(r)}.$ Therefore, we have $v \nleq s_{\sigma(r)}w_{0}^{S\setminus\{\alpha_{r}\}}.$ Otherwise $v(\alpha_{0})\ge s_{\sigma(r)}w_{0}^{S\setminus\{\alpha_{r}\}}(\alpha_{0})=\alpha_{\sigma(r)}.$ Therefore, $v(\alpha_{0})$ is a positive root, a contradiction. Since $v,s_{\sigma(r)}w_{0}^{S\setminus\{\alpha_{r}\}}\in W^{S\setminus\{\alpha_{r}\}},$  we have ${s_{\sigma(r)}}{w_{0}}^{S\setminus\{\alpha_{r}\}}<v.$ Therefore, we have $v=w_{0}^{S\setminus\{\alpha_{r}\}}.$  
	\end{proof}
	
	\section{Preliminaries on non minuscule fundamental weight}
	In this section, we prove a crucial lemma for a non minuscule weight of type $D$ or $E$ associated to the fundamental weight $\alpha_{0}.$ We recall the Dynkin diagram of $D_{n}, E_{6}, E_{7}, E_{8}$ (see\cite[Theorem 11.4, p. 57-58]{Hum1}):
	
	\vspace{1.6cm}
\hspace{1.5cm}
	\setlength{\unitlength}{1.3cm}
	\begin{picture}(10,0)
		\thicklines
		\put(.2,0){\circle*{0.2}}
		\put(0,-0.5){$\alpha_{1}$}
		\put(1.2,0){\circle*{0.2}}
		\put(.2,0){\line(1,0){1}}
		\put(1.2,0){\line(1,0){1}}
		\put(1,-0.5){$\alpha_{2}$}
		\put(2.2,0){\circle*{0.2}}
		\put(2,-0.5){$\alpha_{3}$}
		\put(1.6,0){\line(1,0){1}}
		\put(7.2,0){\circle*{0.2}}
		\put(7,-0.5){$\alpha_{n-3}$}
		\put(8.2,0){\circle*{0.2}}
		\put(7.2,0){\line(1,0){1}}
		\put(8.2,0){\line(1,1){1}}
		\put(8.2,0){\line(1,-1){1}}
		\put(8,-0.5){$\alpha_{n-2}$}
		\put(9.2,-1){\circle*{0.2}}
		\put(9,-1.5){$\alpha_{n}$}
		\put(9.2,1){\circle*{0.2}}
		\put(9,1.5){$\alpha_{n-1}$}
		\put(4,0){\circle*{0.2}}
		\put(3.9,-0.5){$\alpha_{i}$}
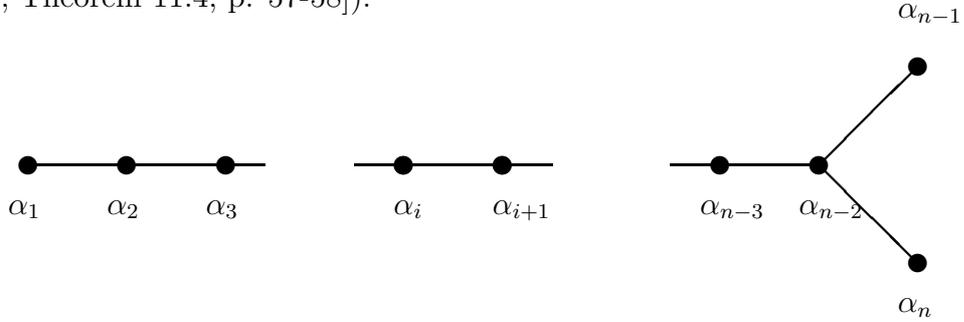
		\put(1.7,-1.9){Figure 1: Dynkin diagram of $D_{n}(n\ge 4).$}
		\put(1.7,-2.4){$\alpha_{0}=\alpha_{1}+2(\alpha_{2}+\alpha_{3}+\cdots +\alpha_{n-3}+\alpha_{n-2})+\alpha_{n-1}+\alpha_{n}=\omega_{2}.$}
		\put(4,0){\line(1,0){1}}
		\put(5,0){\circle*{0.2}}
		\put(4.9,-0.5){$\alpha_{i+1}$}
		\put(4,0){\circle*{0.2}}
		\put(4,0){\line(1,0){1}}
		\put(5,0){\circle*{0.2}}
		\put(3.5,0){\line(1,0){1}}
		\put(4.5,0){\line(1,0){1}}
		\put(6.7,0){\line(1,0){1}}
	\end{picture}
	
	\vspace{5cm}
	\hspace{3.5cm}
	\setlength{\unitlength}{1.4cm}
	\begin{picture}(10,0)
		\thicklines
		\put(.2,0){\circle*{0.2}}
		\put(0,-0.5){$\alpha_{1}$}
		\put(1.2,0){\circle*{0.2}}
		\put(.2,0){\line(1,0){1}}
		\put(1.2,0){\line(1,0){1}}
		\put(1,-0.5){$\alpha_{3}$}
		\put(2.2,0){\circle*{0.2}}
		\put(2,-0.5){$\alpha_{4}$}
		\put(2.2,0){\line(1,0){1}}
		\put(3.2,0){\circle*{0.2}}
		\put(3.1,-0.5){$\alpha_{5}$}
		\put(4.2,0){\circle*{0.2}}
		\put(4,-0.5){$\alpha_{6}$}
		\put(3.2,0){\line(1,0){1}}
		\put(2.2,1){\circle*{0.2}}
		\put(2.1,1.3){$\alpha_{2}$}
		\put(2.2,0){\line(0,1){1}}
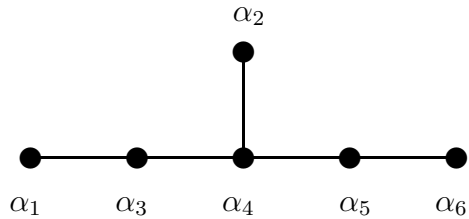
		\put(.3,-1.5){Figure 2: Dynkin diagram of $E_{6}.$}
		\put(.3,-1.9){$\alpha_{0}=\alpha_{1}+2\alpha_{2}+2\alpha_{3}+3\alpha_{4}+2\alpha_{5}+\alpha_{6}=\omega_{2}.$}
	\end{picture}

	\vspace{5cm}
	\hspace{3.5cm}
	\setlength{\unitlength}{1.4cm}
	\begin{picture}(10,0)
		\thicklines
		\put(.2,0){\circle*{0.2}}
		\put(0,-0.5){$\alpha_{1}$}
		\put(1.2,0){\circle*{0.2}}
		\put(.2,0){\line(1,0){1}}
		\put(1.2,0){\line(1,0){1}}
		\put(1,-0.5){$\alpha_{3}$}
		\put(2.2,0){\circle*{0.2}}
		\put(2,-0.5){$\alpha_{4}$}
		\put(2.2,0){\line(1,0){1}}
		\put(3.2,0){\circle*{0.2}}
		\put(3.1,-0.5){$\alpha_{5}$}
		\put(4.2,0){\circle*{0.2}}
		\put(4,-0.5){$\alpha_{6}$}
		\put(3.2,0){\line(1,0){1}}
		\put(2.2,1){\circle*{0.2}}
		\put(2.1,1.3){$\alpha_{2}$}
		\put(2.2,0){\line(0,1){1}}
		\put(5.2,0){\circle*{0.2}}
		\put(4.2,0){\line(1,0){1}}
		\put(5.1,-0.5){$\alpha_{7}$}
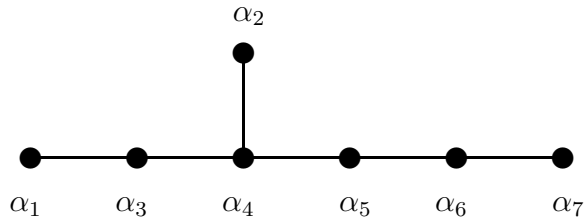
		\put(.3,-1.5){Figure 3: Dynkin diagram of $E_{7}.$}
		\put(.3,-1.9){$\alpha_{0}=2\alpha_{1}+2\alpha_{2}+3\alpha_{3}+4\alpha_{4}+3\alpha_{5}+2\alpha_{6}+\alpha_{7}=\omega_{1}.$}
	\end{picture}
	
\pagebreak	
	\vspace*{1cm}
	\hspace{3.5cm}
	\setlength{\unitlength}{1.4cm}
	\begin{picture}(10,0)
		\thicklines
		\put(.2,0){\circle*{0.2}}
		\put(0,-0.5){$\alpha_{1}$}
		\put(1.2,0){\circle*{0.2}}
		\put(.2,0){\line(1,0){1}}
		\put(1.2,0){\line(1,0){1}}
		\put(1,-0.5){$\alpha_{3}$}
		\put(2.2,0){\circle*{0.2}}
		\put(2,-0.5){$\alpha_{4}$}
		\put(2.2,0){\line(1,0){1}}
		\put(3.2,0){\circle*{0.2}}
		\put(3.1,-0.5){$\alpha_{5}$}
		\put(4.2,0){\circle*{0.2}}
		\put(4,-0.5){$\alpha_{6}$}
		\put(3.2,0){\line(1,0){1}}
		\put(2.2,1){\circle*{0.2}}
		\put(2.1,1.3){$\alpha_{2}$}
		\put(2.2,0){\line(0,1){1}}
		\put(5.2,0){\circle*{0.2}}
		\put(4.2,0){\line(1,0){1}}
		\put(5.1,-0.5){$\alpha_{7}$}
		\put(6.2,0){\circle*{0.2}}
		\put(5.2,0){\line(1,0){1}}
		\put(6.1,-0.5){$\alpha_{8}$}
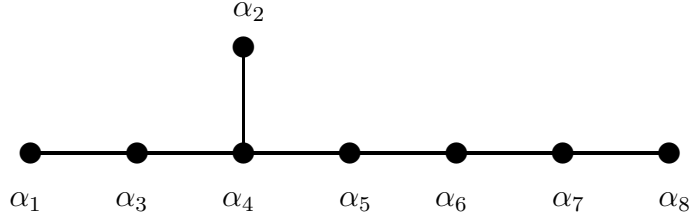
		\put(.5,-1.5){Figure 4: Dynkin diagram of $E_{8}.$}
		\put(.5,-1.9){$\alpha_{0}=2\alpha_{1}+3\alpha_{2}+4\alpha_{3}+6\alpha_{4}+5\alpha_{5}+4\alpha_{6}+3\alpha_{7}+2\alpha_{8}=\omega_{8}.$}
	\end{picture}
	
	\vspace{3cm}
	In type $D_{n},$  $E_{6},$ $E_{7},$ or $E_{8},$ we note that $\alpha_{0}$ is a non minuscule weight. We now prove
	\begin{lem}\label{lemma 3.1}
		Assume that $G$ is of type $D,$ or $E.$
		Then, we have $w_{0, S\setminus \{\alpha_{r}\}}(\alpha_{r})=\alpha_{0} -\alpha_{r},$ where $\alpha_{r}$ is the simple root such that $\alpha_{0}=\omega_{r}.$  
	\end{lem}
	\begin{proof}
		Note that $-\alpha_{r}$ is $L_{S\setminus\{\alpha_{r}\}}$ dominant. Therefore, $w_{0, S\setminus\{\alpha_{r}\}}(-\alpha_{r})$ is $L_{S\setminus\{\alpha_{r}\}}$ negative dominant. 
		
		Next we claim that $\langle w_{0, S\setminus\{\alpha_{r}\}}(-\alpha_{r}), \alpha_{r}\rangle \ge 1.$ On the contrary, if $\langle w_{0, S\setminus\{\alpha_{r}\}}(-\alpha_{r}), \alpha_{r}\rangle \le 0,$ then $w_{0, S\setminus\{\alpha_{r}\}}(-\alpha_{r})$
		is negative dominant for $S.$ Since there is exactly one negative dominant root $-\alpha_{0}$ in the root system, we have $w_{0, S\setminus \{\alpha_{r}\}}(-\alpha_{r})=-\alpha_{0}.$ Therefore, by Lemma \ref{lem 2.3}, $\alpha_{0}=\omega_{r}$ is a minuscule, a contradiction. 
		
		Since $G$ is simply-laced, we have $\langle w_{0, S\setminus\{\alpha_{r}\}}(-\alpha_{r}), \alpha_{r}\rangle=1.$ Thus $\langle s_{r}w_{0, S\setminus\{\alpha_{r}\}}(-\alpha_{r}), \alpha_{r}\rangle=-1.$
		Next we prove the Lemma by studying case by case.
		
		\vspace{.1cm}
		Case I: $G$ is of type $D_{n}.$

		Then we have $\alpha_{0}=\omega_{2}.$ Since $\langle \alpha_{i},\alpha_{2}\rangle=0$ for $i\neq 1,2,3,$ and $w_{0,S\setminus\{\alpha_{2}\}}(-\alpha_{2})$ is $L_{S\setminus\{\alpha_{2}\}}$ negative dominant, we have $\langle s_{2}w_{0, S\setminus\{\alpha_{2}\}}(-\alpha_{2}), \alpha_{i}\rangle\le 0$ for $i\neq 1,3.$ 
		
		Further, $\langle s_{2}w_{0, S\setminus\{\alpha_{2}\}}(-\alpha_{2}), \alpha_{i}\rangle=\langle w_{0, S\setminus\{\alpha_{2}\}}(-\alpha_{2}), \alpha_{2}+\alpha_{i}\rangle$ for $i=1,3.$ By the above discussion we have $\langle w_{0, S\setminus\{\alpha_{2}\}}(-\alpha_{2}), \alpha_{2}\rangle=1.$ Moreover, since  $w_{0,S\setminus \{\alpha_{2}\}}(\alpha_{i})=-\alpha_{i}$ for $i=1,3$, we have $\langle s_{2}w_{0,S\setminus\{\alpha_{2}\}}(-\alpha_{2}), \alpha_{i}\rangle=0$ for $i=1,3.$ Thus $s_{2}w_{0,S\setminus\{\alpha_{2}\}}(-\alpha_{2})$ is a negative dominant for $S.$ Therefore, we have $s_{2}w_{0,S\setminus\{\alpha_{2}\}}(-\alpha_{2})=-\alpha_{0}.$ So, we have $w_{0,S\setminus\{\alpha_{2}\}}(\alpha_{2})=\alpha_{0}-\alpha_{2}.$
		
		\vspace{.2cm}
		
		Case II: $G$ is of type $E_{6}.$

		Then we have $\alpha_{0}=\omega_{2}.$ Since $\langle \alpha_{i},\alpha_{2}\rangle=0$ for $i\neq 2, 4,$ and $w_{0,S\setminus\{\alpha_{2}\}}(-\alpha_{2})$ is $L_{S\setminus \{\alpha_{2}\}}$ negative dominant, we have $\langle s_{2}w_{0, S\setminus\{\alpha_{2}\}}(-\alpha_{2}), \alpha_{i}\rangle\le 0$ for $i\neq 2,4.$ 
		
		We note that $\langle s_{2}w_{0, S\setminus\{\alpha_{2}\}}(-\alpha_{2}), \alpha_{4}\rangle=\langle w_{0, S\setminus\{\alpha_{2}\}}(-\alpha_{2}), \alpha_{2}+\alpha_{4}\rangle.$ By the above discussion we have $\langle w_{0, S\setminus\{\alpha_{2}\}}(-\alpha_{2}), \alpha_{2}\rangle=1.$ Since $w_{0,S\setminus \{\alpha_{2}\}}(\alpha_{4})=-\alpha_{4}$,  $\langle s_{2}w_{0,S\setminus\{\alpha_{2}\}}(-\alpha_{2}), \alpha_{4}\rangle=0.$

		Thus $s_{2}w_{0,S\setminus\{\alpha_{2}\}}(-\alpha_{2})$ is a negative dominant for $S.$ Therefore,  $s_{2}w_{0,S\setminus\{\alpha_{2}\}}(-\alpha_{2})=-\alpha_{0}.$ So, we have $w_{0,S\setminus\{\alpha_{2}\}}(\alpha_{2})=\alpha_{0}-\alpha_{2}.$

		\vspace{.1cm}
		Case III: $G$ is of type $E_{7}.$\\
		Then we have $\alpha_{0}=\omega_{1}.$ Since $\langle \alpha_{i},\alpha_{1}\rangle=0$ for $i\neq 1,3,$  and $w_{0,S\setminus \{\alpha_{1}\}}(-\alpha_{1})$ is $L_{S\setminus \{\alpha_{1}\}}$ negative dominant, we have $\langle s_{\alpha_{1}}w_{0, S\setminus\{\alpha_{1}\}}(-\alpha_{1}), \alpha_{i}\rangle\le 0$ for $i\neq 1,3.$  By the above discussion we have $\langle w_{0, S\setminus\{\alpha_{1}\}}(-\alpha_{1}), \alpha_{1}\rangle=1.$ Therefore,  $\langle s_{1}w_{0, S\setminus\{\alpha_{1}\}}(-\alpha_{1}), \alpha_{1}\rangle=-1.$ On the other hand, $\langle s_{1}w_{0, S\setminus\{\alpha_{1}\}}(-\alpha_{1}), \alpha_{3}\rangle=\langle w_{0, S\setminus\{\alpha_{1}\}}(-\alpha_{1}), \alpha_{1}+\alpha_{3}\rangle.$ Since  $w_{0,S\setminus\{\alpha_{1}\}}(\alpha_{3})=-\alpha_{3}$, we have $\langle s_{1}w_{0,S\setminus\{\alpha_{1}\}}(-\alpha_{1}), \alpha_{3}\rangle=0.$

		Thus $s_{1}w_{0,S\setminus\{\alpha_{1}\}}(-\alpha_{1})$ is a negative dominant for $S.$ Therefore,  $s_{1}w_{0,S\setminus\{\alpha_{1}\}}(-\alpha_{1})=-\alpha_{0}.$ So, we have $w_{0,S\setminus\{\alpha_{1}\}}(\alpha_{1})=\alpha_{0}-\alpha_{1}.$
		
		\vspace{.1cm}
		Case IV: $G$ is of type $E_{8}.$

		Then we have $\alpha_{0}=\omega_{8}.$ Since $\langle \alpha_{i},\alpha_{8}\rangle=0$ for $i\neq 7,8,$ and $w_{0,S\setminus\{\alpha_{8}\}}(-\alpha_{8})$ is $L_{S\setminus\{\alpha_{8}\}}$ negative dominant, we have $\langle s_{8}w_{0, S\setminus\{\alpha_{8}\}}(-\alpha_{8}), \alpha_{i}\rangle\le 0$ for $i\neq 7,8.$ 
		
		Further, $\langle s_{8}w_{0, S\setminus\{\alpha_{8}\}}(-\alpha_{8}), \alpha_{7}\rangle=\langle w_{0, S\setminus\{\alpha_{8}\}}(-\alpha_{8}), \alpha_{8}+\alpha_{7}\rangle.$ By the above discussion we have $\langle w_{0, S\setminus\{\alpha_{8}\}}(-\alpha_{8}), \alpha_{8}\rangle=1.$ 
		
		Moreover, since  $w_{0,S\setminus \{\alpha_{8}\}}(\alpha_{7})=-\alpha_{7}$, we have  $\langle s_{8}w_{0,S\setminus\{\alpha_{8}\}}(-\alpha_{8}), \alpha_{7}\rangle=0.$ \newline Thus $s_{8}w_{0,S\setminus\{\alpha_{8}\}}(-\alpha_{8})$ is a negative dominant for $S.$ Therefore,  $s_{8}w_{0,S\setminus\{\alpha_{8}\}}(-\alpha_{8})=-\alpha_{0}.$ So, we have $w_{0,S\setminus\{\alpha_{8}\}}(\alpha_{8})=\alpha_{0}-\alpha_{8}.$
		
	\end{proof}
	
	\section{$G$ is of type $D_{n}(n\ge 4)$} 
	In this section, we prove the following Proposition:
	\begin{prop}\label{Prop 6.1}
		Assume that $\omega_{i}$ is non minuscule.
		%For any $2\le i\le n-2.$
		Then there exists a Schubert variety  $X_{P_{i}}(w_{i})$ in $G/P_{i}$ such that $P_{i}=Aut^0(X_{P_{i}}(w_{i})).$ 
	\end{prop}
	We recall that there exists a unique element of minimal length $v_{i}$ in $W$ such that $v_{i}^{-1}(\alpha_{0})=-\alpha_{i}$ for all $1\le i\le n$ (see Corollary \ref{corollary 2.1}).
	
	For all $1\le i\le n-2,$  let $u_{i}=s_{i}s_{i+1}\cdots s_{n-2}s_{n-1}s_{n}s_{n-2}s_{n-3}\cdots s_{i+1}s_{i}.$ Note that $u_{i}$ is self inverse i.e., $u_{i}^{-1}=u_{i}.$\\
	Let $w_{2}=u_{2}s_{1},$ $w_{3}=u_{3}v_{2},$ and $w_{i}=u_{i}(s_{i-1}u_{i})\cdots(s_{3}\cdots s_{i-1}u_{i})v_{i-1}$ for all $4\le i\le n-2.$
	
	\begin{lem}\label{lemma 6.2} 
		For  $1\le i\le n-2,$ let $u_{i}$ be as above. Then we have the following: 
		\begin{enumerate}
			\item [(1)]  $u_{i}(\alpha_{j})=\alpha_{j}$ for all $1\le j\le i-2.$
			\item[(2)] $(i).$ $u_{i}(\alpha_{i-1})=\alpha_{i-1}+2\alpha_{i}+2\alpha_{i+1}+\cdots+2\alpha_{n-2}+\alpha_{n-1}+\alpha_{n}$ for all $2\le i\le n-2.$\\ $(ii).$ $u_{i}(\alpha_{i})=-(\alpha_{i}+2\alpha_{i+1}+\cdots+2\alpha_{n-2}+\alpha_{n-1}+\alpha_{n}),$ for all $1\le i\le n-2.$ \\ In particular, we have
			$u_{i}(\alpha_{i-1}+\alpha_{i})=\alpha_{i-1}+\alpha_{i}$ for all $2\le i\le n-2.$
			\item[(3)]$(i).$ $u_{i}(\alpha_{j})=\alpha_{j}$ for all $i+1\le j\le n-2.$\\
			$(ii).$ $u_{i}(\alpha_{n-1})=\alpha_{n},$ and $u_{i}(\alpha_{n})=\alpha_{n-1}$ for all $1\le i\le n-2.$ In particular, we have $u_{i}(\alpha_{n-1}+\alpha_{n})=\alpha_{n-1}+\alpha_{n}$ for all $1\le i\le n-2.$
		\end{enumerate}
	\end{lem}
	\begin{proof}
		Proof of $(1)$: Since $1\le i\le n-2,$ and $1\le j\le i-2,$ we have $\langle \alpha_{j}, \alpha_{k}\rangle=0$ for all $i\le k\le n.$ Therefore,	$u_{i}(\alpha_{j})=s_{i}s_{i+1}\cdots s_{n-2}s_{n-1}s_{n}s_{n-2}s_{n-3}\cdots s_{i}(\alpha_{j})=\alpha_{j}.$
		
		Proof of $(2)$: Follows from the usual calculation using the description of $u_{i}.$  
		
		Proof of $(3)$: We note that $u_{i}=s_{i}\cdots s_{j}u_{j+1}s_{j}s_{j-1}\cdots s_{i}$ for all $i+1\le j\le n-3.$ Since $1\le i\le n-2,$ and $ i+1\le j\le n-2,$ we have
		$s_{j}s_{j-1}\cdots s_{i}(\alpha_{j})=\alpha_{j-1}.$ By $(1)$ we have $u_{j+1}(\alpha_{j-1})=\alpha_{j-1}.$ Further, since $s_{j}s_{j-1}\cdots s_{i}(\alpha_{j})=\alpha_{j-1},$ we have $s_{i}\cdots s_{j}(\alpha_{j-1})=\alpha_{j}.$ Therefore, proof of (3)(i) follows.
		
		Proof of (3)(ii) follows from the usual calculation.
	\end{proof}
	\begin{lem}\label{lemma 6.3}
		Let $v_{i}$ be as above for $1\le i\le n-3$. Then we have the following
		\begin{enumerate}
			\item [(1)] $v_{i}=(s_{2}s_{3}\cdots s_{n-2}s_{n-1}s_{n}s_{n-2}s_{n-3}\cdots s_{i+1})(s_{1}\cdots s_{i})$ for all $1\le i\le n-3.$
			
			\item[(2)] $v_{i}(\alpha_{i+1})=\alpha_{1}+\alpha_{2}+\cdots +\alpha_{i}+\alpha_{i+1}$ for all $1\le i\le n-3.$

			\item[(3)] $w_{i+1}{(\alpha_{i+1})}=\alpha_{1}+\alpha_{2}+\cdots +\alpha_{i}+\alpha_{i+1}$ for all $1\le i\le n-3.$\\ In particular, $w_{i+1}(\alpha_{i+1})$ is a non simple positive root for $1\le i\le n-3.$
		\end{enumerate}
	\end{lem}
	\begin{proof}
		Proof of $(1)$:
		Let ${v'_{i}}=(s_{2}s_{3}\cdots s_{n-2}s_{n-1}s_{n}s_{n-2}s_{n-3}\cdots s_{i+1})(s_{1}\cdots s_{i})$ for all $1\le i\le n-3.$ Then we note that $(s_{2}s_{3}\cdots s_{n-2}s_{n-1}s_{n}s_{n-2}s_{n-3}\cdots s_{i+1})(s_{1}\cdots s_{i})$ is a reduced expression $v'_{i}.$ 
		Further, we observe that ${v'_{i}}^{-1}(\alpha_{0})=-\alpha_{2}.$ Since $\ell({v'}_{2})=\ell(v_{2}),$ by Proposition \ref{Prop 2.1} we have $v_{2}=v'_{2}.$
		
		Proof of $(2)$: Follows from the usual calculation using the description of $v_{i}$ as in $(1).$
		
		Proof of $(3)$: Note that $w_{i+1}=u_{i+1}(s_{i}u_{i+1})\cdots(s_{3}\cdots s_{i}u_{i+1})v_{i}.$ By $(2)$ we have $v_{i}(\alpha_{i+1})=\alpha_{1}+\alpha_{2}+\cdots+\alpha_{i}+\alpha_{i+1}.$ By using Lemma \ref{lemma 6.2} $(1), (2)$  we have $u_{i+1}(\alpha_{1}+\alpha_{2}+\cdots+\alpha_{i}+\alpha_{i+1})=\alpha_{1}+\alpha_{2}+\cdots+\alpha_{i}+\alpha_{i+1}.$ Since $\alpha_{1}+\alpha_{2}+\cdots+\alpha_{i}+\alpha_{i+1}$ is orthogonal to $\alpha_{k}$ for all $3\le k\le i,$ we have $s_{3}\cdots s_{i}(\alpha_{1}+\alpha_{2}+\cdots+\alpha_{i}+\alpha_{i+1})=\alpha_{1}+\alpha_{2}+\cdots+\alpha_{i}+\alpha_{i+1}.$ Similarly, we have $s_{l}\cdots s_{i}(\alpha_{1}+\alpha_{2}+\cdots+\alpha_{i}+\alpha_{i+1})=\alpha_{1}+\alpha_{2}+\cdots+\alpha_{i}+\alpha_{i+1}$ for all $3 \le l\le i.$ Therefore, we have $w_{i}(\alpha_{1}+\alpha_{2}+\cdots+\alpha_{i}+\alpha_{i+1})=\alpha_{1}+\alpha_{2}+\cdots+\alpha_{i}+\alpha_{i+1}.$
		
	\end{proof}
	Recall that $v_{i}=(s_{2}\cdots s_{n-2}s_{n-1}s_{n}s_{n-2}\cdots s_{i+1})s_{1}\cdots s_{i}$ for all $1\le i\le n-3.$
	\begin{lem}\label{lemma 6.5}
		Then $v_{i}$ satisfies the following 
		\begin{enumerate}
			\item[(1)] $v_{i}^{-1}(\alpha_{1})=\alpha_{i}+2\alpha_{i+1}+\cdots +2\alpha_{n-2}+\alpha_{n-1}+\alpha_{n}$ for all $1\le i\le n-3.$
			\vspace*{.1cm}
			\item[(2)] $(i).$ $v_{1}^{-1}(\alpha_{2})=-(\alpha_{1}+2\alpha_{2}+\cdots +2\alpha_{n-2}+\alpha_{n-1}+\alpha_{n}).$\\
			$(ii).$ $v_{i}^{-1}(\alpha_{2})=-(\alpha_{1}+\alpha_{2}+\cdots +2\alpha_{i}+\cdots+2\alpha_{n-2}+\alpha_{n-1}+\alpha_{n})$ for all $2\le i\le n-3.$
			\item[(3)]	
			$v_{i}^{-1}(\alpha_{j})=\alpha_{j-2}$ for all $1\le i\le n-3,$ and $3\le j\le i.$
			
			\item[(4)]  $v_{i}^{-1}(\alpha_{i+1})=\alpha_{i-1}+\alpha_{i}+\alpha_{i+1}$ for all $2\le i\le n-3.$ 
			
			\item [(5)]	
			$(i).$ $v_{i}^{-1}(\alpha_{j})=\alpha_{j}$ for all $1\le i\le n-3,$ and  $i+2\le j\le n-2.$\\
			$(ii).$ $v_{i}^{-1}(\alpha_{n-1})=\alpha_{n},$ and 
			$v_{i}^{-1}(\alpha_{n})=\alpha_{n-1}$ for all $1\le i\le n-3.$ In particular, we have $v_{i}^{-1}(\alpha_{n-1}+\alpha_{n})=\alpha_{n-1}+\alpha_{n}$ for all $1\le i\le n-3.$
		\end{enumerate}
		
	\end{lem}
	\begin{proof}
		Proof of (1): Follows from the usual calculation.
		
		Proof of (2): Follows from the usual calculation.
		
		Proof of (3): Since $3\le j\le i,$ we have $s_{i+1}\cdots s_{n-2}s_{n}s_{n-1}s_{n-2}\cdots s_{2}(\alpha_{j})=\alpha_{j-1}.$ Further, since $3\le j\le i,$
		we have $s_{i}\cdots s_{1}(\alpha_{j-1})=\alpha_{j-2}.$ So, we have $v_{i}^{-1}(\alpha_{j})=\alpha_{j-2}$ for all $3\le j\le i.$
		
		Proof of (4): For  $2\le i\le n-3,$ we have $s_{i+1}\cdots s_{n-2}s_{n}s_{n-1}s_{n-2}\cdots s_{2}(\alpha_{i+1})=\alpha_{i}+\alpha_{i+1}.$ Further,
		we have $s_{i}\cdots s_{1}(\alpha_{i}+\alpha_{i+1})=\alpha_{j-2}.$ So, we have $v_{i}^{-1}(\alpha_{j})=\alpha_{i-1}+\alpha_{i}+\alpha_{i+1}$ for all $2\le i\le n-3.$ 
		
		Proof of (5): For  $i+2\le j\le  n-2,$ we have $s_{i+1}\cdots s_{n-2}s_{n}s_{n-1}s_{n-2}\cdots s_{2}(\alpha_{j})=\alpha_{j}.$ Further, 
		we have $s_{i}\cdots s_{1}(\alpha_{j})=\alpha_{j}$ for all $i+2\le j\le  n-2.$ So, we have $v_{i}^{-1}(\alpha_{j})=\alpha_{j}$ for all $2+i\le j\le n-2.$
		
		We note that $s_{i+1}\cdots s_{n-2}s_{n}s_{n-1}s_{n-2}\cdots s_{2}(\alpha_{n-1})=\alpha_{n}.$ Further, we have $s_{i}\cdots s_{1}(\alpha_{n})=\alpha_{n}.$ Thus we have $v_{i}^{-1}(\alpha_{n-1})=\alpha_{n}.$
		
		Similarly, we note that 
		$s_{i+1}\cdots s_{n-2}s_{n}s_{n-1}s_{n-2}\cdots s_{2}(\alpha_{n})=\alpha_{n-1}.$ Furthermore, we have $s_{i}\cdots s_{1}(\alpha_{n-1})=\alpha_{n-1}.$ So, $v_{i}^{-1}(\alpha_{n})=\alpha_{n-1}.$ 
		
	\end{proof}
	
	\begin{lem}\label{lem 6.5}
		For  $2\le i\le n-2,$ let $w_{i}$ be as above. Then we have $w_{i}^{-1}(\alpha_{j})$ is a positive root for $j\neq i,$ and $w_{i}^{-1}(\alpha_{i})$ is a negative root. 	
	\end{lem}
	\begin{proof}
		Case I: When $i=2.$
		
		Note that $w_{2}=u_{2}s_{1}.$
		
		By Lemma \ref{lemma 6.2}(2), we have  
		$u_{2}^{-1}(\alpha_{1})=\alpha_{1}+2\alpha_{2}+\cdots+ 2\alpha_{n-2}+\alpha_{n-1}+\alpha_{n}.$ Therefore, we have $w_{2}^{-1}(\alpha_{1})=\alpha_{1}+2\alpha_{2}+\cdots+ 2\alpha_{n-2}+\alpha_{n-1}+\alpha_{n}.$ 
		
		By Lemma \ref{lemma 6.2}(3),  for $3\le j\le n,$ $u_{2}^{-1}(\alpha_{j})$ is a positive root whose support does not contain $\alpha_{1}.$ Hence, $w_{2}^{-1}(\alpha_{j})$ is a positive root for all $3\le j\le n.$ 
		
		On the other hand, by Lemma \ref{lemma 6.2}(2),  $u_{2}^{-1}(\alpha_{2})$ is negative of a non simple root. Therefore, $w_{2}^{-1}(\alpha_{2})$ is a negative root.
		\vspace*{.2cm}
		
		Case II: When $i=3.$
		
		Note that $w_{3}=u_{3}v_{2}.$

		By Lemma \ref{lemma 6.2}(2), we have  ${u_{3}}^{-1}(\alpha_{2})=\alpha_{2}+2\alpha_{3}+\cdots +2\alpha_{n-2}+\alpha_{n-1}+\alpha_{n}.$  By using Lemma \ref{lemma 6.5}(2),(3),(4) and (5), we have ${v_{2}}^{-1}(\alpha_{2}+2\alpha_{3}+2\alpha_{4}+\cdots +2\alpha_{n-2}+\alpha_{n-1}+\alpha_{n})=\alpha_{1}.$ Therefore, we have $w_{3}^{-1}(\alpha_{2})=\alpha_{1}.$
		
		By Lemma \ref{lemma 6.2}, for $j=1$ or $4\le j\le n,$ $u_{3}^{-1}(\alpha_{j})$ is a positive root whose support does not contain $\alpha_{2}.$ Therefore, by Lemma \ref{lemma 6.2},  $w_{3}^{-1}(\alpha_{j})=v_{2}^{-1}u_{3}^{-1}(\alpha_{j})$ is a positive root for all $4\le j\le n.$
		
		On the other hand, by Lemma \ref{lemma 6.2}, $u_{3}^{-1}(\alpha_{3})$ is a non simple negative root whose support does not contain $\alpha_{2}.$ Therefore, by Lemma \ref{lemma 6.5}, $w_{3}^{-1}(\alpha_{3})=v_{2}^{-1}u_{3}^{-1}(\alpha_{3})$ is a negative root.
		\vspace{.2cm}
		
		Case III: When $4\le i\le n-2.$
		
		Note that for $4\le i\le n-2,$ we have 
		$w_{i}=u_{i}(s_{i-1}u_{i})\cdots(s_{3}\cdots s_{i-1}u_{i})v_{i-1}.$ 
		
		Since $i\ge 4,$ by using Lemma \ref{lemma 6.2}(1), we have ${(s_{3}\cdots s_{i-1}u_{i})}^{-1}\cdots{(s_{i-1}u_{i})}^{-1}{u_{i}}^{-1}(\alpha_{1})=\alpha_{1}.$ On the other hand, by Lemma \ref{lemma 6.5}(1), we have ${v_{i-1}}^{-1}(\alpha_{1})=\alpha_{i-1}+2\alpha_{i}+\cdots +2\alpha_{n-2}+\alpha_{n-1}+\alpha_{n}.$ Thus we have $w_{i}^{-1}(\alpha_{1})=\alpha_{i-1}+2\alpha_{i}+\cdots +2\alpha_{n-2}+\alpha_{n-1}+\alpha_{n}$ for all $4\le i\le n-2.$

		Since $i\ge 4,$ by using Lemma \ref{lemma 6.2}(1), we have ${(s_{3}\cdots s_{i-1}u_{i})}^{-1}\cdots{(s_{i-1}u_{i})}^{-1}{u_{i}}^{-1}(\alpha_{2})=\alpha_{2}+\alpha_{3}+\cdots +\alpha_{i-1}+2\alpha_{i}+\cdots+ 2\alpha_{n-2}+\alpha_{n-1}+\alpha_{n}.$ On the other hand, by Lemma \ref{lemma 6.5}, we have ${v_{i-1}}^{-1}(\alpha_{2}+\alpha_{3}+\cdots +\alpha_{i-1}+2\alpha_{i}+\cdots+ 2\alpha_{n-2}+\alpha_{n-1}+\alpha_{n})=\alpha_{i-2}.$ Thus we have $w_{i}^{-1}(\alpha_{2})=\alpha_{i-2}$ for all $4\le i\le n-2.$

		For $3\le j\le i-2,$ we have $u_{i}^{-1}(\alpha_{j})=\alpha_{j},$ $(s_{k}\cdots s_{i-1}u_{i})^{-1}(\alpha_{j})=\alpha_{j}$ for all $j+2\le k\le i-1.$ On the other hand, we note that  $(s_{j+1}\cdots s_{i-1}u_{i})^{-1}(\alpha_{j})=u_{i}^{-1}(\alpha_{j}+\alpha_{j+1}+\cdots +\alpha_{i-1}).$ Therefore, by Lemma \ref{lemma 6.2}(1),(2), we have $u_{i}^{-1}(\alpha_{j}+\alpha_{j+1}+\cdots +\alpha_{i-1})=\alpha_{j}+\cdots +\alpha_{i-2}+\alpha_{i-1}+2\alpha_{i}+2\alpha_{i+1}+\cdots +2\alpha_{n-2}+\alpha_{n-1}+\alpha_{n}.$ 
		Now, $(s_{j}s_{j+1}\cdots s_{i-1}u_{i})^{-1}(\alpha_{j}+\cdots +\alpha_{i-2}+\alpha_{i-1}+2\alpha_{i}+2\alpha_{i+1}+\cdots +2\alpha_{n-2}+\alpha_{n-1}+\alpha_{n})=u_{i}^{-1}(\alpha_{i-1}+2\alpha_{i}+2\alpha_{i+1}+\cdots +2\alpha_{n-2}+\alpha_{n-1}+\alpha_{n})=\alpha_{i-1}$(see Lemma \ref{lemma 6.2}(2)).  We have 
		$(s_{j-1}s_{j}s_{j+1}\cdots s_{i-1}u_{i})^{-1}(\alpha_{i-1})=u_{i}^{-1}(\alpha_{i-2})=\alpha_{i-2}$ (see Lemma \ref{lemma 6.2}(1)). Similarly, by recursion we have $(s_{3}\cdots s_{j}s_{j+1}\cdots s_{i-1}u_{i})^{-1}\cdots (s_{i-1}u_{i})^{-1}u_{i}^{-1}(\alpha_{j})=\alpha_{i-j+2}.$ By using Lemma \ref{lemma 6.5}(3), we have $v_{i-1}^{-1}(\alpha_{i-j+2})=\alpha_{i-j}.$  Therefore, we have $w_{i}^{-1}(\alpha_{j})=\alpha_{i-j}$ for all $i\ge 4,$ and $3\le j\le i-2.$

		For $i+1\le j\le n,$ by Lemma \ref{lemma 6.2}(3),  $(u_{i}(s_{i-1}u_{i})\cdots (s_{3}\cdots s_{i-1}u_{i}))^{-1}(\alpha_{j}),$ is a positive root whose support does not contain $\alpha_{2}.$ Thus by Lemma \ref{lemma 6.5}, $v_{i-1}^{-1}(u_{i}(s_{i-1}u_{i})\cdots (s_{3}\cdots s_{i-1}u_{i}))^{-1}(\alpha_{j})$ is a positive root. Therefore, $w_{i}^{-1}(\alpha_{j})$ is a positive root for all $i+1\le j \le n.$

		Now, we consider $i\ge 4,$ and $ j= i-1.$ Then by Lemma \ref{lemma 6.2}(2), we have $u_{i}^{-1}(\alpha_{i-1})=\alpha_{i-1}+2\alpha_{i}+\cdots+2\alpha_{n-2}+\alpha_{n-1}+\alpha_{n}.$ Again, by Lemma \ref{lemma 6.2}, we have $(s_{i-1}u_{i})^{-1}(\alpha_{i-1}+2\alpha_{i}+\cdots+2\alpha_{n-2}+\alpha_{n-1}+\alpha_{n})=\alpha_{i-1}.$ Now, $(s_{i-2}s_{i-1}u_{i})^{-1}(\alpha_{i-1})=u_{i}^{-1}(\alpha_{i-2})=\alpha_{i-2}.$ \\ 
		Similarly, by recursion we have $(s_{3}\cdots s_{j}s_{j+1}\cdots s_{i-1}u_{i})^{-1}(\alpha_{j})=\alpha_{3}.$ By using Lemma \ref{lemma 6.5}(3) we have $v_{i-1}^{-1}(\alpha_{3})=\alpha_{1}.$  Therefore, we have $w_{i}^{-1}(\alpha_{i-1})=\alpha_{1}.$ 
		
		By Lemma \ref{lemma 6.2}(2), we have $u_{i}^{-1}(\alpha_{i})=-(\alpha_{i}+2\alpha_{i+1}+\cdots +2\alpha_{n-2}+\alpha_{n-1}+\alpha_{n}).$ Let $\beta_{i}=\alpha_{i}+2\alpha_{i+1}+\cdots +2\alpha_{n-2}+\alpha_{n-1}+\alpha_{n}.$ Therefore, by Lemma \ref{lemma 6.2}(2), and Lemma \ref{lemma 6.2}(3), we have $(s_{i-1}u_{i})^{-1}(-\beta_{i})=u_{i}^{-1}(-\beta_{i} -\alpha_{i-1})=-(\alpha_{i-1}+\beta_{i}).$ Then by using Lemma \ref{lemma 6.2}(1),(2), and (3), we have $(s_{i-2}s_{i-1}u_{i})^{-1}(-(\alpha_{i-1}+\beta_{i}))=-u_{i}^{-1}(\alpha_{i-2}+\alpha_{i-1}+\beta_{i})=-(\alpha_{i-2}+\alpha_{i-1}+\beta_{i}).$ Thus by recursion we have $(s_{4}\cdots s_{i-1}u_{i})^{-1}\cdots (s_{i-1}u_{i})^{-1}u_{i}^{-1}(\alpha_{i})=-(\alpha_{4}+\cdots +\alpha_{i-1}+\beta_{i}).$ Therefore, $(s_{3}\cdots s_{i-1}u_{i})^{-1}(-(\alpha_{4}+\cdots +\alpha_{i-1}+\beta_{i}))=-u_{i}^{-1}(\alpha_{3}+\alpha_{4}+\cdots +\alpha_{i-1}+\beta_{i})=-(\alpha_{3}+\alpha_{4}+\cdots +\alpha_{i-1}+\beta_{i}).$ Since support of $\alpha_{3}+\alpha_{4}+\cdots +\alpha_{i-1}+\beta_{i}$ does not contain $\alpha_{2},$ by Lemma \ref{lemma 6.5},  
		$v_{i-1}^{-1}(-(\alpha_{3}+\alpha_{4}+\cdots +\alpha_{i-1}+\beta_{i})$ is a negative root. Thus $w_{i}^{-1}(\alpha_{i})$ is a negative root.

		Therefore, combining Case I, Case II, and Case III proof of the lemma follows.
	\end{proof}

	\begin{proof}[{\bf Proof of Proposition \ref{Prop 6.1}:}]
		We note that $\omega_{i}$ is not minuscule  if and only if  $2\le i\le n-2.$ Fix an integer  $2\le i\le n-2.$ Then by Lemma \ref{lemma 6.3}(3), we conclude  that $w_{i}(\alpha_{i})$ is a non simple positive root.
		On the other hand, by  Lemma \ref{lem 6.5},  $w_{i}^{-1}(\alpha_{i})$ is a negative root, and $w_{i}^{-1}(\alpha_{j})$ is a positive root for $j\neq i.$ Thus $P_{i}$ is the stabilizer of $X_{P_{i}}(w_{i})$ in $G/P_{i}$ for the natural action of $G.$ Further, since $s_{2}\nleq u_{i}(s_{i-1}u_{i})\cdots (s_{3}\cdots s_{i-1}u_{i}),$ and $\alpha_{0}=\omega_{2},$ we have $(s_{3}\cdots s_{i-1}u_{i})^{-1}\cdots (s_{i-1}u_{i})^{-1}u_{i}^{-1}(\alpha_{0})=\alpha_{0}.$ Therefore,  $w_{i}^{-1}(\alpha_{0})=v_{i-1}^{-1}(\alpha_{0})$ is a negative root. Therefore, by using  Theorem \ref{thm1} the natural homomorphism $P_{i}\longrightarrow Aut^0(X_{P_{i}}(w_{i}))$ is an isomorphism of algebraic groups. 
	\end{proof}

	\section{$G$ is of type $E_{6}$}
	In this section, we prove our result for $E_{6}.$ Note that by Lemma \ref{lem 3.2}, the non-minuscule weights are $\omega_{i}$ where $i=2,3,4,5.$ In this section, we prove the following:
	
	\begin{prop}\label{proposition 5.4}
		Assume that $\omega_{i}$ is non-minuscule. Then there exists a Schubert variety $X_{P_{4}}(w_{i})$ in $G/P_{4}$ such that $P_{i}=Aut^0(X_{P_{4}}(w_{i})).$
	\end{prop}
	
	Recall that there exists a unique element $v_{2}\in W$ of smallest length such that $v_{2}^{-1}(\alpha_{0})=-\alpha_{2}$ (see Corollary \ref{corollary 2.1}).
	
	\begin{lem}\label{lemma 5.3}
		Let $v_{2}\in W$ be as above. Then we have the following
		\begin{enumerate}
			\item [(1)] $v_{2}=s_{2}s_{4}s_{5}s_{3}s_{6}s_{4}s_{1}s_{3}s_{5}s_{4}s_{2}.$ In particular, we have  $v_{2}^{-1}=v_{2}.$
			
			\item[(2)]$v_{2}(\alpha_{1})=\alpha_{5},$ $v_{2}(\alpha_{3})=\alpha_{6},$ $v_{2}(\alpha_{4})=\alpha_{2}+\alpha_{3}+2\alpha_{4}+\alpha_{5},$ and $v_{2}(\omega_{2}-\alpha_{2})=\omega_{2}-\alpha_{2}.$
			
			\item[(3)] (i). $w_{0, S\setminus \{\alpha_{2}\}} v_{2}(\alpha_{4})=\omega_{2}-(\alpha_{2}+\alpha_{3}+2\alpha_{4}+\alpha_{5})=\alpha_{1}+\alpha_{2}+\alpha_{3}+\alpha_{4}+\alpha_{5}+\alpha_{6}.$ \\
			(ii). $w_{0, S\setminus \{\alpha_{2}, \alpha_{i}\}}w_{0, S\setminus \{\alpha_{2}\}}v_{2}(\alpha_{4})$ is a non simple positive root for $i=3,4,5.$
		\end{enumerate}
	\end{lem} 
	\begin{proof}
		Proof of (1): Let ${v'}_{2}=s_{2}s_{4}s_{5}s_{3}s_{6}s_{4}s_{1}s_{3}s_{5}s_{4}s_{2}.$
		Note that ${{v'}_{2}}^{-1}(\alpha_{0})=-\alpha_{2}.$ Since $\ell({v'}_{2})=\ell(v_{2}),$ by Corollary \ref{corollary 2.1}, we have $v_{2}=v'_{2}.$
		
		Proof of (2):  Follows from the usual calculation. 
		
		Proof of (3)(i): By (2), we have $v_{2}(\alpha_{4})=\alpha_{2}+\alpha_{3}+2\alpha_{4}+\alpha_{5}.$  We observe that the Dynkin subdiagram of  $E_{6}$ corresponding to  $S\setminus \{\alpha_{2}\}$ is of type $A_{5}.$ Since $w_{0,S\setminus\{\alpha_{2}\}}$ is the longest element of $W_{S\setminus\{\alpha_{2}\}},$ we have $w_{0,S\setminus \{\alpha_{2}\}}(\alpha_{1})=-\alpha_{6},$ $w_{0,S\setminus \{\alpha_{2}\}}(\alpha_{3})=-\alpha_{5},$ and $w_{0,S\setminus\{\alpha_{2}\}}(\alpha_{4})=-\alpha_{4}.$ Thus by using Lemma \ref{lemma 3.1}, 
		we have $w_{0,S\setminus\{\alpha_{2}\}}v_{2}(\alpha_{4})=\omega_{2}-(\alpha_{2}+\alpha_{3}+2\alpha_{4}+\alpha_{5}).$
		
		(ii). By (3)(i), we have $w_{0, S\setminus \{\alpha_{2}\}}v_{2}(\alpha_{4})=\alpha_{1}+\alpha_{2}+\alpha_{3}+\alpha_{4}+\alpha_{5}+\alpha_{6}.$ Since support of $w_{0,S\setminus\{\alpha_{2}\}}v_{2}(\alpha_{4})$ contains $\alpha_{i},$ $1\le i\le 6,$ proof of (3)(ii) follows. 
	\end{proof}

	\begin{lem}\label{lemma 5.2} We have the following	
		\begin{enumerate}	
			\item[(1)] $(i).$ $w_{0,S\setminus\{ \alpha_{2}, \alpha_{3}\}}(\alpha_{2})=\alpha_{2}+\alpha_{4}+\alpha_{5}+\alpha_{6}.$\\
			$(ii).$ $w_{0,S\setminus\{\alpha_{2}\}}w_{S\setminus\{ \alpha_{2}, \alpha_{3}\}}(\alpha_{2})=\omega_{2}-(\alpha_{1}+\alpha_{2}+\alpha_{3}+\alpha_{4}).$\\
			$(iii).$ $v_{2}w_{0,S\setminus\{\alpha_{2}\}}w_{S\setminus\{ \alpha_{2}, \alpha_{3}\}}(\alpha_{2})=\omega_{2}-(2\alpha_{2}+\alpha_{3}+2\alpha_{4}+2\alpha_{5}+\alpha_{6})=\alpha_{1}+\alpha_{3}+\alpha_{4}.$
		\end{enumerate}

		\begin{enumerate}
			
			\item [(2)] $(i).$ 
			$w_{S\setminus\{ \alpha_{2}, \alpha_{4}\}}(\alpha_{2})=\alpha_{2}.$\\
			$(ii).$ $w_{0,S\setminus\{\alpha_{2}\}}w_{S\setminus\{ \alpha_{2}, \alpha_{\alpha_{4}}\}}(\alpha_{2})=\omega_{2}-\alpha_{2}.$\\
			$(iii).$ $v_{2}w_{0,S\setminus\{\alpha_{2}\}}w_{S\setminus\{ \alpha_{2}, \alpha_{4}\}}(\alpha_{2})=\omega_{2}-\alpha_{2}.$
		\end{enumerate}
		\begin{enumerate}
			\item [(3)] $(i).$ $w_{S\setminus\{ \alpha_{2}, \alpha_{5}\}}(\alpha_{2})=\alpha_{1}+\alpha_{2}+\alpha_{3}+\alpha_{4}.$\\
			$(ii).$ $w_{0,S\setminus\{\alpha_{2}\}}w_{S\setminus\{ \alpha_{2}, \alpha_{5}\}}(\alpha_{2})=\omega_{2}-(\alpha_{2}+\alpha_{4}+\alpha_{5}+\alpha_{6}).$\\
			$(iii).$ $v_{2}w_{0,S\setminus\{\alpha_{2}\}}w_{S\setminus\{ \alpha_{2}, \alpha_{5}\}}(\alpha_{2})=\omega_{2}-(\alpha_{1}+2\alpha_{2}+2\alpha_{3}+2\alpha_{4}+\alpha_{5})=\alpha_{4}+\alpha_{5}+\alpha_{6}.$

			\item[(4)] $w_{S\setminus\{ \alpha_{2}, \alpha_{i}\}}(\alpha_{i})=\alpha_{1}+\alpha_{3}+\alpha_{4}+\alpha_{5}+\alpha_{6}$  for $i=3,4,5.$
		\end{enumerate}
		
	\end{lem}
	\begin{proof}
		Proof of (1): (i) Note that the Dynkin subdiagram of $E_{6},$ corresponding to the subset $S\setminus \{\alpha_{3}\}$ of $S$ (see figure 2):

		
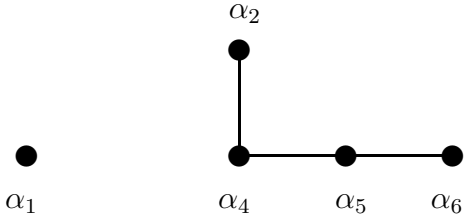
\vspace*{2cm}
		\begin{picture}(10,0)
			\thicklines
			\put(.2,0){\circle*{0.2}}
			\put(0,-0.5){$\alpha_{1}$}
			%\put(0.2,0){\line(1,0){1}}
			%\put(1.2,0){\circle*{0.2}}
			%\put(1.2,0){\line(1,0){1}}
			%\put(1,-0.5){$\alpha_{3}$}
			\put(2.2,0){\circle*{0.2}}
			\put(2,-0.5){$\alpha_{4}$}
			\put(2.2,0){\line(1,0){1}}
			\put(3.2,0){\circle*{0.2}}
			\put(3.1,-0.5){$\alpha_{5}$}
			\put(4.2,0){\circle*{0.2}}
			\put(4,-0.5){$\alpha_{6}$}
			\put(3.2,0){\line(1,0){1}}
			\put(2.2,1){\circle*{0.2}}
			\put(2.1,1.3){$\alpha_{2}$}
			\put(2.2,0){\line(0,1){1}}
			
		\end{picture}
		\vspace{1cm}
		
		Let $I=S\setminus\{\alpha_{3}\}.$
		Now we observe that $w_{0, S\setminus \{\alpha_{2}, \alpha_{3}\}}(\alpha_{2})=w_{0,I\setminus\{\alpha_{2}\}}(\alpha_{2}).$We note that the connected component of the Dynkin subdiagram associated to $I,$ containing $\alpha_{2}$ is of type $A_{4}.$  Since $\alpha_{2}$ is  minuscule  in type $A_{4},$  by Lemma \ref{lem 2.3}, we have $w_{0,I\setminus\{\alpha_{2}\}}(\alpha_{2})=\alpha_{2}+\alpha_{4}+\alpha_{5}+\alpha_{6}.$
		
		(ii). Note that the  Dynkin subdiagram of  $E_{6}$ corresponding to  $S\setminus \{\alpha_{2}\}$  is of type $A_{5}.$ Since $w_{0,S\setminus\{\alpha_{2}\}}$ is the longest element of $W_{S\setminus\{\alpha_{2}\}},$ we have $w_{0,S\setminus \{\alpha_{2}\}}(\alpha_{1})=-\alpha_{6},$ $w_{0,S\setminus \{\alpha_{2}\}}(\alpha_{3})=-\alpha_{5},$ and $w_{0,S\setminus\{\alpha_{2}\}}(\alpha_{4})=-\alpha_{4}.$ By using the above discussion together with Lemma \ref{lemma 3.1} proof of (ii) follows.
		
		(iii). By (ii), we have $w_{0,S\setminus\{\alpha_{2}\}}w_{0,S\setminus\{\alpha_{2},\alpha_{3}\}}(\alpha_{2})=\omega_{2}-(\alpha_{1}+\alpha_{2}+\alpha_{3}+\alpha_{4}).$ Therefore, by using Lemma \ref{lemma 5.3}(1),(2), and  we get $v_{2}w_{0,S\setminus\{\alpha_{2}\}}w_{0, S\setminus\{\alpha_{2},\alpha_{3}\}}=\omega_{2}-(2\alpha_{2}+\alpha_{3}+2\alpha_{4}+2\alpha_{5}+\alpha_{6}).$ Further, simplifying we have $v_{2}w_{S\setminus\{\alpha_{2}\}}w_{0, S\setminus\{\alpha_{2},\alpha_{3}\}}=\alpha_{1}+\alpha_{3}+\alpha_{4}.$
		
		Proof of (2)(i): Next we consider the Dynkin subdiagram of $E_{6},$ corresponding to the subset $S\setminus\{\alpha_{4}\}$ of $S:$
		
		
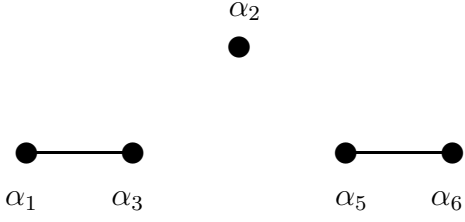
\vspace*{2cm}
		\begin{picture}(10,0)
			\thicklines
			\put(.2,0){\circle*{0.2}}
			\put(0,-0.5){$\alpha_{1}$}
			\put(0.2,0){\line(1,0){1}}
			\put(1.2,0){\circle*{0.2}}
			%\put(1.2,0){\line(1,0){1}}
			\put(1,-0.5){$\alpha_{3}$}
			%\put(2.2,0){\circle*{0.2}}
			%\put(2,-0.5){$\alpha_{4}$}
			%\put(2.2,0){\line(1,0){1}}
			\put(3.2,0){\circle*{0.2}}
			\put(3.1,-0.5){$\alpha_{5}$}
			\put(4.2,0){\circle*{0.2}}
			\put(4,-0.5){$\alpha_{6}$}
			\put(3.2,0){\line(1,0){1}}
			\put(2.2,1){\circle*{0.2}}
			\put(2.1,1.3){$\alpha_{2}$}
			%\put(2.2,0){\line(0,1){1}}
		\end{picture}
		\vspace{1cm}
		
		Let $I=S\setminus \{\alpha_{4}\}.$ Then we observe that $w_{0,S\setminus \{\alpha_{2}, \alpha_{4}\}}(\alpha_{2})=w_{0,I\setminus\{\alpha_{2}\}}(\alpha_{2}).$ Further, we note that the connected component of the Dynkin subdiagram associated to $I,$ containing $\alpha_{2}$ is of type $A_{1}.$ Since $\alpha_{2}$ is  minuscule  in type $A_{1}$, by Lemma \ref{lem 2.3}, we have $w_{0,I\setminus\{\alpha_{2}\}}(\alpha_{2})=\alpha_{2}.$ 
		
		(ii). Proof is similar to the proof of (1)(ii).
		
		(iii). By using Lemma \ref{lemma 5.3}(2) and (ii), proof of (iii) follows.

		Proof of (3)(i): 
		Next we consider the Dynkin subdiagram of $E_{6},$ corresponding to the subset $S\setminus\{\alpha_{5}\}$ of $S:$ 
		
		
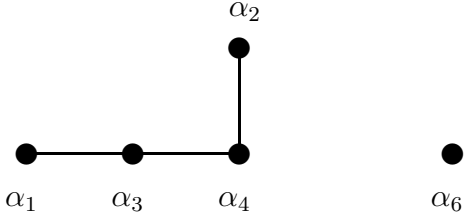
\vspace*{2cm}
		\begin{picture}(10,0)
			\thicklines
			\put(.2,0){\circle*{0.2}}
			\put(0,-0.5){$\alpha_{1}$}
			\put(0.2,0){\line(1,0){1}}
			\put(1.2,0){\circle*{0.2}}
			\put(1.2,0){\line(1,0){1}}
			\put(1,-0.5){$\alpha_{3}$}
			\put(2.2,0){\circle*{0.2}}
			\put(2,-0.5){$\alpha_{4}$}
			%\put(2.2,0){\line(1,0){1}}
			%\put(3.2,0){\circle*{0.2}}
			%\put(3.1,-0.5){$\alpha_{5}$}
			\put(4.2,0){\circle*{0.2}}
			\put(4,-0.5){$\alpha_{6}$}
			%\put(3.2,0){\line(1,0){1}}
			\put(2.2,1){\circle*{0.2}}
			\put(2.1,1.3){$\alpha_{2}$}
			\put(2.2,0){\line(0,1){1}}
		\end{picture}
		\vspace{1cm}
		
		Let $I=S\setminus \{\alpha_{5}\}.$ Then we observe that $w_{0,S\setminus \{\alpha_{2}, \alpha_{5}\}}(\alpha_{2})=w_{0,I\setminus\{\alpha_{2}\}}(\alpha_{2}).$ Further, we note that the connected component of the Dynkin subdiagram associated to $I,$ containing $\alpha_{2}$ is of type $A_{4}.$ Since $\alpha_{2}$ is  minuscule  in type $A_{4}$, by Lemma \ref{lem 2.3}, we have $w_{0,I\setminus\{\alpha_{2}\}}(\alpha_{2})=\alpha_{1}+\alpha_{2}+\alpha_{3}+\alpha_{4}.$ 
		
		(ii). Proof is similar to the proof of (1)(ii). 
		
		(iii). By using Lemma \ref{lemma 5.3}(2) and (ii), proof of (iii) follows.

		Proof of (4):
		For fix an $i$ in $\{3,4,5\},$ let $I=S\setminus \{\alpha_{2}\}.$ Then we have $w_{0, S\setminus\{\alpha_{2}, \alpha_{i}\}}(\alpha_{i})=w_{0,I\setminus\{\alpha_{i}\}}(\alpha_{i}).$ Then we observe that the Dynkin subdiagram associated to $I$ is of type $A_{5}.$ Since $\alpha_{i}$ is  minuscule  in type $A_{5},$ by Lemma \ref{lem 2.3}, we have $w_{0, I\setminus\{\alpha_{i}\}}(\alpha_{i})=\alpha_{1}+\alpha_{3}+\alpha_{4}+\alpha_{5}+\alpha_{6}.$
		
	\end{proof}
	\begin{lem}\label{lemma 5.5}Let $w_{i}=w_{0,S\setminus\{\alpha_{2}, \alpha_{i}\}}w_{0, S\setminus\{\alpha_{2}\}}v_{2}$ for $i\neq 1,6.$  Then we have $w_{i}^{-1}(\alpha_{j})$ is a positive root for $j\neq i,$ and $w_{i}^{-1}(\alpha_{i})$ is negative root.
	\end{lem}
	\begin{proof}
		Note that for $i=2$ we have $w_{2}=v_{2}.$ Then by Lemma \ref{lemma 5.3}(1),(2), we are done. For $i\neq 1,2,6,$ let  $w_{0,S\setminus \{\alpha_{2}, \alpha_{i}\}}(\alpha_{i})=\beta.$ Then $\beta$ is a positive root whose support does not contain $\alpha_{2}.$  Since $w_{0, S\setminus \{\alpha_{2}\}}$ is the longest element of $W_{S\setminus \{\alpha_{2}\}},$  $w_{0,S\setminus \{\alpha_{2}\}}(\beta)$ is a negative root whose support does not contain $\alpha_{2}.$ Further, since the support of $w_{0,S\setminus \{\alpha_{2}\}}(\beta)$ does not contain $\alpha_{2},$ by Lemma \ref{lemma 5.3}(2), we have $v_{2}w_{0,S\setminus \{\alpha_{2}\}}(\beta)$ is a negative root. Hence, we have $w_{i}^{-1}(\alpha_{i})$ is a negative root for $i\neq 1,2,6.$ 
		
		On the other hand, for $i\neq 1,2,6$ and $j\neq 2,i,$  $w_{0,S\setminus \{\alpha_{2}, \alpha_{i}\}}(\alpha_{j})=-\alpha_{k},$ where $\alpha_{k}$ is a simple root different from $\alpha_{2}.$ Therefore, $w_{0,S\setminus \{\alpha_{2}\}}(-\alpha_{k})=\alpha_{l}$ for some $l\neq 2.$ Further, since $\alpha_{l}$ is different from $\alpha_{2},$ by Lemma \ref{lemma 5.3}(2), $v_{2}(\alpha_{l})$ is a positive root. Hence, $w_{i}^{-1}(\alpha_{j})$ is a positive root when $i\neq 1,2,6$ and $j\neq 2,i.$ Moreover, by Lemma \ref{lemma 5.2}, we have  $w_{i}^{-1}(\alpha_{2})$ is a positive root for $i\neq 2,1,6.$ Hence, we have $w_{i}^{-1}(\alpha_{j})$ is a positive root for $i\neq 1,6$ and $j\neq i.$
	\end{proof}
	
	\begin{proof}[{\bf Proof of Proposition \ref{proposition 5.4}:}]
		We note that $\omega_{i}$ is not minuscule if and only if $i\neq 1,6.$ Assume that $\omega_{i}$ is  not minuscule. Recall that  $w_{i}=w_{0,S\setminus \{\alpha_{2}, \alpha_{i}\}} w_{0, S\setminus \{\alpha_{2}\}}v_{2}.$ By Lemma \ref{lemma 5.3}(3), we conclude  that $w_{i}(\alpha_{4})$ is a non simple positive root.
		On the other hand, by Lemma \ref{lemma 5.5}, $w_{i}^{-1}(\alpha_{i})$ is a negative root and $w_{i}^{-1}(\alpha_{j})$ is a positive root for $j\neq i.$ Thus $P_{i}$ is the stabilizer of $X_{P_{4}}(w_{i})$ in $G.$ Since $\alpha_{0}=\omega_{2},$ and $v_{2}^{-1}(\alpha_{0})=-\alpha_{2},$  $w_{i}^{-1}(\alpha_{0})=v_{2}^{-1}(\alpha_{0})$ (as$~ w_{0,S\setminus\{\alpha_{2}\}}w_{0,S\setminus\{\alpha_{2},\alpha_{i}\}}(\alpha_{0})=\alpha_{0}$) is a negative root. Therefore, by using Theorem \ref{thm1} the natural homomorphism  $P_{i}\longrightarrow Aut^0(X_{P_{4}}(w_{i}))$ is an isomorphism of algebraic groups.  
	\end{proof}
	\section{$G$ is of type $E_{7}$}
	We  note that by Lemma \ref{lem 3.2}, $\omega_{i}$ is not minuscule of $E_{7}$ if and only if  $i\neq7.$
	In this section, our goal is to prove the following proposition:
	\begin{prop}\label{Proposition 5.7} Assume that $\omega_{i}$ is not minuscule. Then there exists a Schubert variety $X_{P_{3}}(w_{i})$ in $G/P_{3}$ such that $P_{i}=Aut^0(X_{P_{3}}(w_{i})).$	
	\end{prop}

	We recall that by Corollary \ref{corollary 2.1}, there exists a unique element $v_{4}$ in $W$ of minimal length such that $v_{4}^{-1}(\alpha_{0})=-\alpha_{4}$. We note that $w_{0,S\setminus \{\alpha_{1}\}}$ is the longest element of the Dynkin subdiagram of $E_{7}$ corresponding to the subset  $S\setminus\{\alpha_{1}\}$ of $S.$ Thus we have $w_{0, S\setminus \{\alpha_{1}\}}(\alpha_{i})=-\alpha_{i}$ for $i\neq1.$
	\begin{lem}\label{lemma 5.8}
		Let $v_{4}$ be as above. Then we have the following
		\begin{enumerate}
			\item [(1)] $v_{4}=s_{1}s_{3}s_{4}s_{5}s_{2}s_{4}s_{3}s_{6}s_{5}s_{4}s_{1}s_{2}s_{3}s_{7}s_{6}s_{5}s_{4}.$
		\end{enumerate}
		\begin{enumerate}
			\item[(2)]
			$(i).$ $v_{4}(\alpha_{3})=\alpha_{1}+\alpha_{2}+\alpha_{3}+2\alpha_{4}+2\alpha_{5}+\alpha_{6}+\alpha_{7}.$\\
			$(ii).$ 
			$w_{0, S\setminus \{\alpha_{1}\}}v_{4}(\alpha_{3})=\alpha_{1}+\alpha_{2}+2\alpha_{3}+\alpha_{4}+\alpha_{5}+\alpha_{6}.$\\
			$(iii).$ 
			$w_{0, S\setminus \{\alpha_{1}, \alpha_{i}\}}w_{0, S\setminus \{\alpha_{1}\}}v_{4}(\alpha_{3})$ is a non simple positive root for $i\neq 7.$
		\end{enumerate}
		\begin{enumerate}
			\item [(3)] $v_{4}^{-1}(\alpha_{1})=-(\alpha_{1}+2\alpha_{2}+2\alpha_{3}+4\alpha_{4}+3\alpha_{5}+2\alpha_{6}+\alpha_{7}),$ $v_{4}^{-1}(\alpha_{2})=\alpha_{7},$\\ $v_{4}^{-1}(\alpha_{3})=\alpha_{2}+\alpha_{4}+\alpha_{5},$ $v_{4}^{-1}(\alpha_{4})=\alpha_{6},$ $v_{4}^{-1}(\alpha_{5})=\alpha_{3}+\alpha_{4}+\alpha_{5},$ $v_{4}^{-1}(\alpha_{6})=\alpha_{1},$ $v_{4}^{-1}(\alpha_{7})=\alpha_{2}+\alpha_{3}+\alpha_{4}.$
			
		\end{enumerate}

	\end{lem}
	\begin{proof}Proof of (1): Let ${v'}_{4}=s_{1}s_{3}s_{4}s_{5}s_{2}s_{4}s_{3}s_{6}s_{5}s_{4}s_{1}s_{2}s_{3}s_{7}s_{6}s_{5}s_{4}.$
		Note that ${{v'}_{4}}^{-1}(\alpha_{0})=-\alpha_{4}.$ Since $\ell({v'}_{4})=\ell(v_{4}),$ by Corollary \ref{corollary 2.1}, we have $v_{4}=v'_{4}.$
		
		Proof of (2): (i) Follows from the usual calculation.

		(ii) Follows from (2)(i), and using Lemma \ref{lemma 3.1}.
		Since support of $w_{0, S\setminus\{\alpha_{1}\}}v_{4}(\alpha_{3})$ contains $\alpha_{i}$ $i\neq 7,$ proof of (2)(iii) follows.

		Proof of (3): Follows from the usual calculation by using the description of $v_{4}$ as in (1). 
		
	\end{proof}
	
	\begin{lem}\label{lemma 5.9} We have the following 
		\begin{enumerate}
			\item[(1)]
			$(i)$ $w_{0, S\setminus \{\alpha_{1}, \alpha_{2}\}}(\alpha_{1})=\alpha_{1}+\alpha_{3}+\alpha_{4}+\alpha_{5}+\alpha_{6}+\alpha_{7}.$\\
			$(ii)$	$w_{0,S\setminus \{\alpha_{1}\}}w_{0, S\setminus \{\alpha_{1}, \alpha_{2}\}}(\alpha_{1})=\omega_{1}-(\alpha_{1}+\alpha_{3}+\alpha_{4}+\alpha_{5}+\alpha_{6}+\alpha_{7}).$\\
			$(iii)$ ${v_{4}}^{-1}w_{0,S\setminus \{\alpha_{1}\}}w_{0, S\setminus \{\alpha_{1}, \alpha_{2}\}}(\alpha_{1})=\alpha_{5}+\alpha_{6}+\alpha_{7}.$
			
			\vspace{.3cm}
			\item[(2)]  
			$(i)$ $w_{0, S\setminus \{\alpha_{1}, \alpha_{3}\}}(\alpha_{1})=\alpha_{1}.$\\
			$(ii)$	$w_{0,S\setminus \{\alpha_{1}\}}w_{0, S\setminus \{\alpha_{1}, \alpha_{3}\}}(\alpha_{1})=\omega_{1}-\alpha_{1}.$\\
			$(iii)$ ${v_{4}}^{-1}w_{0,S\setminus \{\alpha_{1}\}}w_{0, S\setminus \{\alpha_{1}, \alpha_{3}\}}(\alpha_{1})=\alpha_{1}+2\alpha_{2}+2\alpha_{3}+3\alpha_{4}+3\alpha_{5}+2\alpha_{6}+\alpha_{7}.$

			\item[(3)]  
			$(i)$ $w_{0, S\setminus \{\alpha_{1}, \alpha_{4}\}}(\alpha_{1})=\alpha_{1}+\alpha_{3}.$\\
			$(ii)$	$w_{0,S\setminus \{\alpha_{1}\}}w_{0, S\setminus \{\alpha_{1}, \alpha_{4}\}}(\alpha_{1})=\omega_{1}-(\alpha_{1}+\alpha_{3}).$\\
			$(iii)$ ${v_{4}}^{-1}w_{0,S\setminus \{\alpha_{1}\}}w_{0, S\setminus \{\alpha_{1}, \alpha_{4}\}}(\alpha_{1})=\alpha_{1}+\alpha_{2}+2\alpha_{3}+2\alpha_{4}+2\alpha_{5}+2\alpha_{6}+\alpha_{7}.$

			\item[(4)]  
			$(i)$ $w_{0, S\setminus \{\alpha_{1}, \alpha_{5}\}}(\alpha_{1})=\alpha_{1}+\alpha_{2}+\alpha_{3}+\alpha_{4}.$\\
			$(ii)$	$w_{0,S\setminus \{\alpha_{1}\}}w_{0, S\setminus \{\alpha_{1}, \alpha_{5}\}}(\alpha_{1})=\omega_{1}-(\alpha_{1}+\alpha_{2}+\alpha_{3}+\alpha_{4}).$\\
			$(iii)$ ${v_{4}}^{-1}w_{0,S\setminus \{\alpha_{1}\}}w_{0, S\setminus \{\alpha_{1}, \alpha_{5}\}}(\alpha_{1})=\alpha_{1}+\alpha_{2}+2\alpha_{3}+2\alpha_{4}+2\alpha_{5}+\alpha_{6}.$
			
			\item[(5)]  
			$(i)$ $w_{0, S\setminus \{\alpha_{1}, \alpha_{6}\}}(\alpha_{1})=\alpha_{1}+2\alpha_{3}+2\alpha_{4}+\alpha_{2}+\alpha_{5}.$\\
			$(ii)$	$w_{0,S\setminus \{\alpha_{1}\}}w_{0, S\setminus \{\alpha_{1}, \alpha_{6}\}}(\alpha_{1})=\omega_{1}-(\alpha_{1}+2\alpha_{3}+2\alpha_{4}+\alpha_{2}+\alpha_{5}).$\\
			$(iii)$ ${v_{4}}^{-1}w_{0,S\setminus \{\alpha_{1}\}}w_{0, S\setminus \{\alpha_{1}, \alpha_{6}\}}(\alpha_{1})=\alpha_{1}+\alpha_{3}.$
		\end{enumerate}
	\end{lem}
	
	\begin{proof}
		Proof of (1): (i) We consider the Dynkin subdiagram of $E_{7},$ corresponding to the subset $S\setminus \{\alpha_{2}\}$ of $S$ (see Figure 3):

		
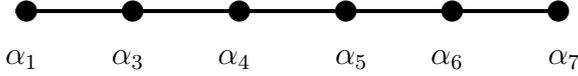
\vspace*{1cm}
		\begin{picture}(10,0)
			\thicklines
			\put(.2,0){\circle*{0.2}}
			\put(0,-0.5){$\alpha_{1}$}
			\put(0.1,0){\line(1,0){1}}
			\put(1.2,0){\circle*{0.2}}
			\put(1.2,0){\line(1,0){1}}
			\put(1,-0.5){$\alpha_{3}$}
			\put(2.2,0){\circle*{0.2}}
			\put(2,-0.5){$\alpha_{4}$}
			\put(2.2,0){\line(1,0){1}}
			\put(3.2,0){\circle*{0.2}}
			\put(3.1,-0.5){$\alpha_{5}$}
			\put(4.2,0){\circle*{0.2}}
			\put(4,-0.5){$\alpha_{6}$}
			\put(3.2,0){\line(1,0){1}}
			%\put(2.2,1){\circle*{0.2}}
			%\put(2.1,1.3){$\alpha_{2}$}
			\put(5.2,0){\circle*{0.2}}
			\put(4.2,0){\line(1,0){1}}
			\put(5.1,-0.5){$\alpha_{7}$}
		\end{picture}
		\vspace{1cm}
		
		Let $I=S\setminus\{\alpha_{2}\}.$
		Now we observe that $w_{0, S\setminus \{\alpha_{1}, \alpha_{2}\}}(\alpha_{1})=w_{0,I\setminus\{\alpha_{1}\}}(\alpha_{1}).$
		We note that the connected component of the Dynkin subdiagram associated to $I,$ containing $\alpha_{1}$ is of type $A_{6.}$ Since $\alpha_{1}$ is minuscule  in type $A_{6}$, by Lemma \ref{lem 2.3}, we have $w_{0,I\setminus\{\alpha_{1}\}}(\alpha_{1})=\alpha_{1}+\alpha_{3}+\alpha_{4}+\alpha_{5}+\alpha_{6}+\alpha_{7}.$
		
		(ii). Note that the Dynkin subdiagram of $E_{7},$ corresponding to $S\setminus \{\alpha_{1}\}$ is of type $D_{6}.$ Since $w_{0, S\setminus \{\alpha_{1}\}}$ is the longest element of $W_{ S\setminus\{\alpha_{1}\}},$ we have $w_{0, S\setminus \{\alpha_{1}\}}(\alpha_{i})=-\alpha_{i}$  for all $i\neq 1.$ Therefore, from the above discussion together with Lemma \ref{lemma 3.1} proof of (1)(ii) follows.
		
		(iii). By (1)(ii) we have $w_{0,S\setminus\{\alpha_{1}\}}w_{0,S\setminus\{\alpha_{1},\alpha_{2}\}}(\alpha_{1})=\omega_{1}-(\alpha_{1}+\alpha_{3}+\alpha_{4}+\alpha_{5}+\alpha_{6}+\alpha_{7}).$ Therefore, by using Lemma \ref{lemma 5.8}(3), we get ${v_{4}}^{-1}w_{0,S\setminus\{\alpha_{1}\}}w_{0, S\setminus\{\alpha_{1},\alpha_{2}\}}(\alpha_{1})=-\alpha_{4}+(\alpha_{1}+2\alpha_{2}+2\alpha_{3}+4\alpha_{4}+3\alpha_{5}+2\alpha_{6}+\alpha_{7})-(\alpha_{2}+\alpha_{4}+\alpha_{5})-\alpha_{6}-(\alpha_{3}+\alpha_{4}+\alpha_{5})-\alpha_{1}-(\alpha_{2}+\alpha_{3}+\alpha_{4})=\alpha_{5}+\alpha_{6}+\alpha_{7}.$
		
		Proof of (2): (i) Next we consider the Dynkin subdiagram of $E_{7},$ corresponding to the  subset $S\setminus \{\alpha_{3}\}$ of $S:$ 
		
		
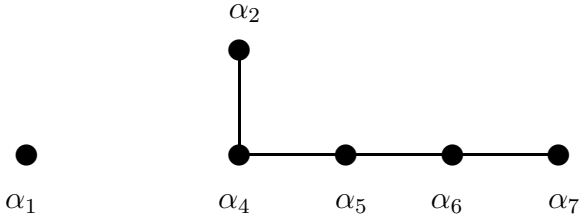
\vspace*{2cm}
		\begin{picture}(10,0)
			\thicklines
			\put(.2,0){\circle*{0.2}}
			\put(0,-0.5){$\alpha_{1}$}
			%\put(0.2,0){\line(1,0){1}}
			%\put(1.2,0){\circle*{0.2}}
			%\put(1.2,0){\line(1,0){1}}
			%\put(1,-0.5){$\alpha_{3}$}
			\put(2.2,0){\circle*{0.2}}
			\put(2,-0.5){$\alpha_{4}$}
			\put(2.2,0){\line(1,0){1}}
			\put(3.2,0){\circle*{0.2}}
			\put(3.1,-0.5){$\alpha_{5}$}
			\put(4.2,0){\circle*{0.2}}
			\put(4,-0.5){$\alpha_{6}$}
			\put(3.2,0){\line(1,0){1}}
			\put(2.2,1){\circle*{0.2}}
			\put(2.1,1.3){$\alpha_{2}$}
			\put(2.2,0){\line(0,1){1}}
			\put(5.2,0){\circle*{0.2}}
			\put(4.2,0){\line(1,0){1}}
			\put(5.1,-0.5){$\alpha_{7}$}
		\end{picture}
		\vspace{1cm}
		
		Let $I=S\setminus \{\alpha_{3}\}.$ Then we observe that $w_{0,S\setminus \{\alpha_{1}, \alpha_{3}\}}(\alpha_{1})=w_{0,I\setminus\{\alpha_{1}\}}(\alpha_{1}).$ Further, we note that the connected component of the Dynkin subdiagram associated to $I,$ containing $\alpha_{1}$ is of type $A_{1}.$ Since $\alpha_{1}$ is minuscule  in type $A_{1},$ by Lemma \ref{lem 2.3}, we have $w_{0,I\setminus\{\alpha_{1}\}}(\alpha_{1})=\alpha_{1}$ 
		
		(ii). Similar to the proof of (1)(ii).  
		
		(iii).  By (2)(ii) we have $w_{0,S\setminus\{\alpha_{1}\}}w_{0,S\setminus\{\alpha_{1},\alpha_{2}\}}(\alpha_{1})=\omega_{1}-\alpha_{1}.$ Therefore, by using Lemma \ref{lemma 5.8}(3), we get ${v_{4}}^{-1}w_{0,S\setminus\{\alpha_{1}\}}w_{0, S\setminus\{\alpha_{1},\alpha_{2}\}}(\alpha_{1})=-\alpha_{4}+(\alpha_{1}+2\alpha_{2}+2\alpha_{3}+4\alpha_{4}+3\alpha_{5}+2\alpha_{6}+\alpha_{7})=\alpha_{1}+2\alpha_{2}+2\alpha_{3}+3\alpha_{4}+3\alpha_{5}+2\alpha_{6}+\alpha_{7}.$

		Proof of (3): (i) 
		Next we consider the Dynkin subdiagram of $E_{7},$ corresponding to the subset $S\setminus \{\alpha_{4}\}$ of $S:$ 
		
		
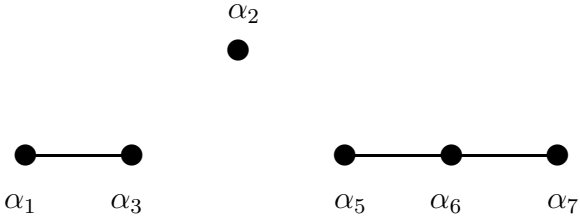
\vspace{2cm}
		\begin{picture}(10,0)
			\thicklines
			\put(.2,0){\circle*{0.2}}
			\put(0,-0.5){$\alpha_{1}$}
			\put(0.2,0){\line(1,0){1}}
			\put(1.2,0){\circle*{0.2}}
			%\put(1.2,0){\line(1,0){1}}
			\put(1,-0.5){$\alpha_{3}$}
			%\put(2.2,0){\circle*{0.2}}
			%\put(2,-0.5){$\alpha_{4}$}
			%\put(2.2,0){\line(1,0){1}}
			\put(3.2,0){\circle*{0.2}}
			\put(3.1,-0.5){$\alpha_{5}$}
			\put(4.2,0){\circle*{0.2}}
			\put(4,-0.5){$\alpha_{6}$}
			\put(3.2,0){\line(1,0){1}}
			\put(2.2,1){\circle*{0.2}}
			\put(2.1,1.3){$\alpha_{2}$}
			%\put(2.2,0){\line(0,1){1}}
			\put(5.2,0){\circle*{0.2}}
			\put(4.2,0){\line(1,0){1}}
			\put(5.1,-0.5){$\alpha_{7}$}
		\end{picture}
		\vspace{1cm}
		
		Let $I=S\setminus \{\alpha_{4}\}.$ Then we observe that $w_{0,S\setminus \{\alpha_{1}, \alpha_{4}\}}(\alpha_{1})=w_{0,I\setminus\{\alpha_{1}\}}(\alpha_{1}).$ Further, we note that the connected component of the Dynkin subdiagram associated to $I,$  containing $\alpha_{1}$ is of type $A_{2}.$ Since $\alpha_{1}$ is minuscule  in type $A_{1}$, by Lemma \ref{lem 2.3}, we have $w_{0,I\setminus\{\alpha_{1}\}}(\alpha_{1})=\alpha_{1}+\alpha_{3}$ 
		
		(ii). Similar to the proof of (1)(ii).
		
		(iii). By (3)(ii) we have $w_{0, S\setminus\{\alpha_{1}\}}w_{0, S\setminus\{\alpha_{1}, \alpha_{3}\}}(\alpha_{1})=\omega_{1}-(\alpha_{1}+\alpha_{3}).$ Therefore, by using  Lemma \ref{lemma 5.8}(3), we get $v_{4}^{-1}w_{0,S\setminus \{\alpha_{1}\}}w_{0, S\setminus \{\alpha_{1}, \alpha_{3}\}}(\alpha_{1})=-\alpha_{4}+(\alpha_{1} +2\alpha_{2}+2\alpha_{3}+4\alpha_{4}+3\alpha_{5}+2\alpha_{6}+\alpha_{7})-(\alpha_{2}+\alpha_{4}+\alpha_{5})=\alpha_{1}+\alpha_{2}+2\alpha_{3}+2\alpha_{4}+2\alpha_{5}+2\alpha_{6}+\alpha_{7}.$

		Proof of (4)(i): 
		Next we consider the Dynkin subdiagram of $E_{7},$ corresponding to the subset $S\setminus\{\alpha_{5}\}$ of $S:$ 
		
		
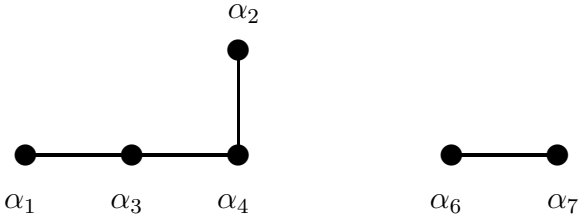
\vspace*{2cm}
		\begin{picture}(10,0)
			\thicklines
			\put(.2,0){\circle*{0.2}}
			\put(0,-0.5){$\alpha_{1}$}
			\put(0.2,0){\line(1,0){1}}
			\put(1.2,0){\circle*{0.2}}
			\put(1.2,0){\line(1,0){1}}
			\put(1,-0.5){$\alpha_{3}$}
			\put(2.2,0){\circle*{0.2}}
			\put(2,-0.5){$\alpha_{4}$}
			%\put(2.2,0){\line(1,0){1}}
			%\put(3.2,0){\circle*{0.2}}
			%\put(3.1,-0.5){$\alpha_{5}$}
			\put(4.2,0){\circle*{0.2}}
			\put(4,-0.5){$\alpha_{6}$}
			%\put(3.2,0){\line(1,0){1}}
			\put(2.2,1){\circle*{0.2}}
			\put(2.1,1.3){$\alpha_{2}$}
			\put(2.2,0){\line(0,1){1}}
			\put(5.2,0){\circle*{0.2}}
			\put(4.2,0){\line(1,0){1}}
			\put(5.1,-0.5){$\alpha_{7}$}
		\end{picture}
		\vspace{1cm}
		
		Let $I=S\setminus \{\alpha_{5}\}.$ Then we observe that $w_{0,S\setminus \{\alpha_{1}, \alpha_{5}\}}(\alpha_{1})=w_{0,I\setminus\{\alpha_{1}\}}(\alpha_{1}).$ Further, we note that the connected component of the Dynkin subdiagram associated to $I,$ containing $\alpha_{1}$ is of type $A_{4}.$ Since $\alpha_{1}$ is minuscule  in type $A_{4},$ by Lemma \ref{lem 2.3}, we have $w_{0,I\setminus\{\alpha_{1}\}}(\alpha_{1})=\alpha_{1}+\alpha_{3}+\alpha_{4}+\alpha_{2}.$ 
		
		(ii). Similar to the proof of (1)(i).
		
		(iii). By (4)(ii) we have $w_{0, S\setminus\{\alpha_{1}\}}w_{0, S\setminus\{\alpha_{1}, \alpha_{3}\}}(\alpha_{1})=\omega_{1}-(\alpha_{1}+\alpha_{3}+\alpha_{4}+\alpha_{2}).$ Therefore, by using  Lemma \ref{lemma 5.8}(3), we get $v_{4}^{-1}w_{0,S\setminus \{\alpha_{1}\}}w_{0, S\setminus \{\alpha_{1}, \alpha_{3}\}}(\alpha_{1})=-\alpha_{4}+(\alpha_{1} +2\alpha_{2}+2\alpha_{3}+4\alpha_{4}+3\alpha_{5}+2\alpha_{6}+\alpha_{7})-(\alpha_{2}+\alpha_{4}+\alpha_{5})-\alpha_{6}-\alpha_{7}=\alpha_{1}+\alpha_{2}+2\alpha_{3}+2\alpha_{4}+2\alpha_{5}+\alpha_{6}.$

		Proof of (5)(i): 
		Next we consider the Dynkin subdiagram of $E_{7},$ corresponding to the subset $S\setminus\{\alpha_{6}\}$ of $S:$ 
		
		
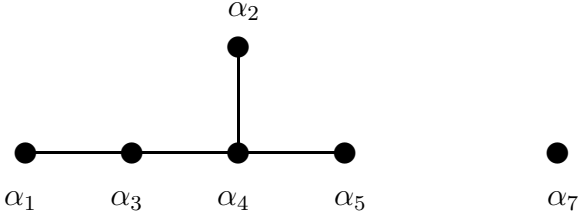
\vspace*{2cm}
		\begin{picture}(10,0)
			\thicklines
			\put(.2,0){\circle*{0.2}}
			\put(0,-0.5){$\alpha_{1}$}
			\put(0.2,0){\line(1,0){1}}
			\put(1.2,0){\circle*{0.2}}
			\put(1.2,0){\line(1,0){1}}
			\put(1,-0.5){$\alpha_{3}$}
			\put(2.2,0){\circle*{0.2}}
			\put(2,-0.5){$\alpha_{4}$}
			\put(2.2,0){\line(1,0){1}}
			\put(3.2,0){\circle*{0.2}}
			\put(3.1,-0.5){$\alpha_{5}$}
			%\put(4.2,0){\circle*{0.2}}
			%\put(4,-0.5){$\alpha_{6}$}
			%\put(3.2,0){\line(1,0){1}}
			\put(2.2,1){\circle*{0.2}}
			\put(2.1,1.3){$\alpha_{2}$}
			\put(2.2,0){\line(0,1){1}}
			\put(5.2,0){\circle*{0.2}}
			%\put(4.2,0){\line(1,0){1}}
			\put(5.1,-0.5){$\alpha_{7}$}
		\end{picture}
		\vspace{1cm}
		
		Let $I=S\setminus \{\alpha_{6}\}.$ Then we observe that $w_{0,S\setminus \{\alpha_{1}, \alpha_{4}\}}(\alpha_{1})=w_{0,I\setminus\{\alpha_{1}\}}(\alpha_{1}).$ Further, we note that the connected component of the Dynkin subdiagram associated to $I,$ containing $\alpha_{1}$ is of type $D_{5}.$ Since $\alpha_{1}$ is minuscule  in type $D_{5}$, by Lemma \ref{lem 2.3}, we have $w_{0,I\setminus\{\alpha_{1}\}}(\alpha_{1})=\alpha_{1}+2(\alpha_{3}+\alpha_{4})+\alpha_{2}+\alpha_{5}.$ 
		
		(ii).  Similar to the proof of (1)(ii).
		
		(iii). By (5)(ii) we have $w_{0, S\setminus\{\alpha_{1}\}}w_{0, S\setminus\{\alpha_{1}, \alpha_{3}\}}(\alpha_{1})=\omega_{1}-(\alpha_{1}+2\alpha_{3}+2\alpha_{4}+\alpha_{2}+\alpha_{5}).$ Therefore, by using  Lemma \ref{lemma 5.8}(3), we get $v_{4}^{-1}w_{0,S\setminus \{\alpha_{1}\}}w_{0, S\setminus \{\alpha_{1}, \alpha_{3}\}}(\alpha_{1})=-\alpha_{4}+(\alpha_{1} +2\alpha_{2}+2\alpha_{3}+4\alpha_{4}+3\alpha_{5}+2\alpha_{6}+\alpha_{7})-2(\alpha_{2}+\alpha_{4}+\alpha_{5})-2\alpha_{6}-\alpha_{7}-(\alpha_{3}+\alpha_{4}+\alpha_{5})=\alpha_{1}+\alpha_{3}.$
	\end{proof}
	\begin{lem}\label{lemma 5.10}Let $w_{i}=w_{0,S\setminus\{\alpha_{1}, \alpha_{i}\}}w_{0, S\setminus\{\alpha_{1}\}}v_{4}$ for $i\neq 7.$  Then we have $w_{i}^{-1}(\alpha_{j})$ is a positive root for $j\neq i,$ and $w_{i}^{-1}(\alpha_{i})$ is a negative root.
	\end{lem}
	\begin{proof}
		Note that for $i=1,$ we have $w_{1}=v_{4}.$ Then by Lemma \ref{lemma 5.8}(3), we are done. For $i\neq 1,$ let  $w_{0,S\setminus \{\alpha_{1}, \alpha_{i}\}}(\alpha_{i})=\beta.$ Then $\beta$ is a positive root whose supports does not contain $\alpha_{1}.$ Since $w_{0, S\setminus \{\alpha_{1}\}}$ is the longest element of $W_{S\setminus \{\alpha_{1}\}}$ (Weyl group of  type $D_{6}$), we have  $w_{0,S\setminus \{\alpha_{1}\}}(\alpha_{i})=-\alpha_{i}$ for all $i\neq 1.$ Thus we have $w_{0,S\setminus \{\alpha_{1}\}}(\beta)=-\beta.$ Further, since the support of $\beta$ does not contain $\alpha_{1},$ by Lemma \ref{lemma 5.8}(3) we have $v_{4}^{-1}(-\beta)$ is a negative root. Hence we have $w_{i}^{-1}(\alpha_{i})$ is a negative root for $i\neq 1$. 
		
		On the other hand, for $i\neq1,7$ and $j\neq 1,i,$  $w_{0,S\setminus \{\alpha_{1}, \alpha_{i}\}}(\alpha_{j})=-\alpha_{k},$ where $\alpha_{k}$ is a simple root different from $\alpha_{1}.$ Therefore, $w_{0,S\setminus \{\alpha_{1}\}}(-\alpha_{k})=\alpha_{k}.$ Further, since $\alpha_{k}$ is different from $\alpha_{1},$ by Lemma \ref{lemma 5.8}(3), $v_{4}^{-1}(\alpha_{k})$ is a positive root. Hence $w_{i}^{-1}(\alpha_{j})$ is a positive root when $i\neq 1,7$ and $j\neq 1,i.$ Moreover, by Lemma \ref{lemma 5.9}, we have  $w_{i}^{-1}(\alpha_{1})$ is a positive root for $i\neq 1,7.$ Hence, we have $w_{i}^{-1}(\alpha_{j})$ is a positive root for $i\neq 7$ and $j\neq i.$
	\end{proof}
	
	\begin{proof}[{\bf Proof of Proposition \ref{Proposition 5.7}:}]
		We note that $\omega_{i}$ is not minuscule if and only if $i\neq 7.$ Assume that $\omega_{i}$ is  not minuscule. Recall that $w_{i}=w_{0,S\setminus \{\alpha_{1}, \alpha_{i}\}} w_{0, S\setminus \{\alpha_{1}\}} v_{4}.$ By Lemma \ref{lemma 5.8}(2) and (3), we conclude  that $w_{i}(\alpha_{3})$ is a non simple positive root.
		On the other hand, by Lemma \ref{lemma 5.10},  $w_{i}^{-1}(\alpha_{i})$ is a negative root and $w_{i}^{-1}(\alpha_{j})$ is a positive root for $j\neq i.$ Thus $P_{i}$ is the stabilizer of $X_{P_{3}}(w_{i})$ in $G.$ Since $\alpha_{0}=\omega_{1},$ and $v_{4}^{-1}(\alpha_{0})$ is a negative root,  $w_{i}^{-1}(\alpha_{0})=v_{4}^{-1}(\alpha_{0})$ (as$~w_{0,S\setminus\{\alpha_{1}\}}w_{0,S\setminus\{\alpha_{1}, \alpha_{i}\}}(\alpha_{0})=\alpha_{0}$) is a negative root. Therefore, by using Theorem \ref{thm1} the natural homomorphism  $P_{i}\longrightarrow Aut^0(X_{P_{3}}(w_{i}))$ is an isomorphism of  algebraic groups. 
	\end{proof}

	\section{$G$ is of type $E_{8}$}
	By Lemma \ref{lem 3.2}, we note that all simple roots are non-minuscule. In this section, our goal is to prove the following proposition
	\begin{prop}\label{proposition 5.12} 
		Assume that $\omega_{i}$ is non minuscule.
		Then  there exists $X_{P_{7}}(w_{i})$ in $G/P_{7}$ such that 
		$P_{i}=Aut^0(X_{P_{7}}(w_{i})).$
		
	\end{prop}

	We recall that there exists a unique element $v_{6}$ in $W$ of minimal length such that $v_{6}^{-1}(\alpha_{0})=-\alpha_{6}$ (see Corollary \ref{corollary 2.1}). Further, we observe that $w_{0,S\setminus \{\alpha_{8}\}}$ is the longest element of the Dynkin subdiagram removing the node $\alpha_{8},$ which is of type $E_{7}.$ Thus  $w_{0, S\setminus \{\alpha_{8}\}}(\alpha_{i})=-\alpha_{i}$ for all $i\neq8.$ We use these observations very frequently.
	\begin{lem}\label{lemma 4.13}
		Let $v_{6}$ be as above. Then we have the following
		\begin{enumerate}
			\item [(1)] $v_{6}=s_{8}s_{7}s_{6}s_{5}s_{4}s_{2}s_{3}s_{1}s_{4}s_{3}s_{5}s_{6}s_{4}s_{2}s_{5}s_{7}s_{4}s_{6}s_{5}s_{3}s_{4}s_{2}s_{8}s_{7}s_{1}s_{3}s_{4}s_{5}s_{6}.$
		\end{enumerate}
		\begin{enumerate}
			\item[(2)]
			$(i).$ $v_{6}(\alpha_{7})=\alpha_{1}+2\alpha_{2}+3\alpha_{3}+4\alpha_{4}+3\alpha_{5}+2\alpha_{6}+\alpha_{7}+\alpha_{8}.$\\
			$(ii).$ 
			$w_{0, S\setminus \{\alpha_{8}\}}v_{6}(\alpha_{7})=\alpha_{1}+\alpha_{2}+\alpha_{3}+2\alpha_{4}+2\alpha_{5}+2\alpha_{6}+2\alpha_{7}+\alpha_{8}.$\\
			$(iii).$ 
			$w_{0, S\setminus \{\alpha_{8}, \alpha_{i}\}}w_{0, S\setminus \{\alpha_{8}\}}v_{6}(\alpha_{7})$ is a non simple positive root for $i=1,2,3,4,5,6,7,8.$
		\end{enumerate}
		\begin{enumerate}
			\item [(3)] $v_{6}^{-1}(\alpha_{1})=\alpha_{8},$ $v_{6}^{-1}(\alpha_{2})=\alpha_{2},$ $v_{6}^{-1}(\alpha_{3})=\alpha_{5}+\alpha_{6}+\alpha_{7},$ $v_{6}^{-1}(\alpha_{4})=\alpha_{4},$ $v_{6}^{-1}(\alpha_{5})=\alpha_{3},$ $v_{6}^{-1}(\alpha_{6})=\alpha_{1},$ $v_{6}^{-1}(\alpha_{7})=\alpha_{2}+\alpha_{3}+2\alpha_{4}+2\alpha_{5}+\alpha_{6},$ $v_{6}^{-1}(\alpha_{8})=-(2\alpha_{1}+3\alpha_{2}+4\alpha_{3}++6\alpha_{4}+5\alpha_{5}+4\alpha_{6}+2\alpha_{7}+\alpha_{8}).$
			
		\end{enumerate}

	\end{lem}
	\begin{proof}Proof of (1): Let ${v'}_{6}=s_{8}s_{7}s_{6}s_{5}s_{4}s_{2}s_{3}s_{1}s_{4}s_{3}s_{5}s_{6}s_{4}s_{2}s_{5}s_{7}s_{4}s_{6}s_{5}s_{3}s_{4}s_{2}s_{8}s_{7}s_{1}s_{3}s_{4}s_{5}s_{6}.$
		Note that ${{v'}_{6}}^{-1}(\alpha_{0})=-\alpha_{6}.$ Since $\ell({v'}_{6})=\ell(v_{6}),$ by Proposition \ref{Prop 2.1} we have $v_{6}=v'_{6}.$
		
		Proof of (2): (i) Follows from the usual calculation using description of $v_{6}$ as in (1).\\(ii) follows from (i), and using Lemma \ref{lemma 3.1}.\\
		Since support of $w_{0, S\setminus\{\alpha_{1}\}}v_{4}(\alpha_{3})$ contains $\alpha_{i}$ $i=1,2,3,4,5,6,7,8$ (iii) follows.
		
		Proof of (3): Follows from the usual calculation using the description of $v_{6}$ as in $(i).$ 
		
	\end{proof}
	\begin{lem}\label{lemma 5.14} We have the following 
		\begin{enumerate}
			\item[(1)]
			$(i)$ $w_{0, S\setminus \{\alpha_{1}, \alpha_{8}\}}(\alpha_{8})=\alpha_{2}+\alpha_{3}+2\alpha_{4}+2\alpha_{5}+2\alpha_{6}+2\alpha_{7}+\alpha_{8}.$\\
			$(ii)$	$w_{0,S\setminus \{\alpha_{8}\}}w_{0, S\setminus \{\alpha_{1}, \alpha_{8}\}}(\alpha_{8})=\omega_{8}-(\alpha_{2}+\alpha_{3}+2\alpha_{4}+2\alpha_{5}+2\alpha_{6}+2\alpha_{7}+\alpha_{8}).$\\
			$(iii)$ ${v_{6}}^{-1}w_{0,S\setminus \{\alpha_{8}\}}w_{0, S\setminus \{\alpha_{1}, \alpha_{8}\}}(\alpha_{8})=\alpha_{7}+\alpha_{8}.$

			\item[(2)]  
			$(i)$ $w_{0, S\setminus \{\alpha_{2}, \alpha_{8}\}}(\alpha_{8})=\alpha_{1}+\alpha_{3}+\alpha_{4}+\alpha_{5}+\alpha_{6}+\alpha_{7}+\alpha_{8}.$\\
			$(ii)$	$w_{0,S\setminus \{\alpha_{8}\}}w_{0, S\setminus \{\alpha_{2}, \alpha_{8}\}}(\alpha_{8})=\omega_{8}-(\alpha_{1}+\alpha_{3}+\alpha_{4}+\alpha_{5}+\alpha_{6}+\alpha_{7}+\alpha_{8}).$\\
			$(iii)$ ${v_{6}}^{-1}w_{0,S\setminus \{\alpha_{8}\}}w_{0, S\setminus \{\alpha_{2}, \alpha_{8}\}}(\alpha_{8})=\alpha_{1}+2\alpha_{2}+2\alpha_{3}+3\alpha_{4}+2\alpha_{5}+\alpha_{6}+\alpha_{7}.$

			\item[(3)]  
			$(i)$ $w_{0, S\setminus \{\alpha_{3}, \alpha_{8}\}}(\alpha_{8})=\alpha_{2}+\alpha_{4}+\alpha_{5}+\alpha_{6}+\alpha_{7}+\alpha_{8}.$\\
			$(ii)$	$w_{0,S\setminus \{\alpha_{8}\}}w_{0, S\setminus \{\alpha_{3}, \alpha_{8}\}}(\alpha_{8})=\omega_{8}-(\alpha_{2}+\alpha_{4}+\alpha_{5}+\alpha_{6}+\alpha_{7}+\alpha_{8}).$\\
			$(iii)$ ${v_{6}}^{-1}w_{0,S\setminus \{\alpha_{8}\}}w_{0, S\setminus \{\alpha_{3}, \alpha_{8}\}}(\alpha_{8})=\alpha_{1}+\alpha_{2}+2\alpha_{3}+3\alpha_{4}+3\alpha_{5}+2\alpha_{6}+2\alpha_{7}+\alpha_{8}.$

			\item[(4)]  
			$(i)$ $w_{0, S\setminus \{\alpha_{4}, \alpha_{8}\}}(\alpha_{8})=\alpha_{5}+\alpha_{6}+\alpha_{7}+\alpha_{8}.$\\
			$(ii)$	$w_{0,S\setminus \{\alpha_{8}\}}w_{0, S\setminus \{\alpha_{4}, \alpha_{8}\}}(\alpha_{8})=\omega_{8}-(\alpha_{5}+\alpha_{6}+\alpha_{7}+\alpha_{8}).$\\
			$(iii)$ ${v_{6}}^{-1}w_{0,S\setminus \{\alpha_{8}\}}w_{0, S\setminus \{\alpha_{4}, \alpha_{8}\}}(\alpha_{8})=\alpha_{1}+2\alpha_{2}+2\alpha_{3}+4\alpha_{4}+3\alpha_{5}+2\alpha_{6}+2\alpha_{7}+\alpha_{8}.$
			
			\item[(5)]  
			$(i)$ $w_{0, S\setminus \{\alpha_{5}, \alpha_{8}\}}(\alpha_{8})=\alpha_{6}+\alpha_{7}+\alpha_{8}.$\\
			$(ii)$	$w_{0,S\setminus \{\alpha_{8}\}}w_{0, S\setminus \{\alpha_{5}, \alpha_{8}\}}(\alpha_{8})=\omega_{8}-(\alpha_{6}+\alpha_{7}+\alpha_{8}).$\\
			$(iii)$ ${v_{6}}^{-1}w_{0,S\setminus \{\alpha_{8}\}}w_{0, S\setminus \{\alpha_{5}, \alpha_{8}\}}(\alpha_{8})=\alpha_{1}+2\alpha_{2}+3\alpha_{3}+4\alpha_{4}+3\alpha_{5}+2\alpha_{6}+2\alpha_{7}+\alpha_{8}.$

			\item[(6)]  
			$(i)$ $w_{0, S\setminus \{\alpha_{6}, \alpha_{8}\}}(\alpha_{8})=\alpha_{7}+\alpha_{8}.$\\
			$(ii)$	$w_{0,S\setminus \{\alpha_{8}\}}w_{0, S\setminus \{\alpha_{5}, \alpha_{8}\}}(\alpha_{8})=\omega_{8}-(\alpha_{7}+\alpha_{8}).$\\
			$(iii)$ ${v_{6}}^{-1}w_{0,S\setminus \{\alpha_{8}\}}w_{0, S\setminus \{\alpha_{5}, \alpha_{8}\}}(\alpha_{8})=2\alpha_{1}+2\alpha_{2}+3\alpha_{3}+4\alpha_{4}+3\alpha_{5}+2\alpha_{6}+2\alpha_{7}+\alpha_{8}.$

			\item[(7)]  
			$(i)$ $w_{0, S\setminus \{\alpha_{7}, \alpha_{8}\}}(\alpha_{8})=\alpha_{8}.$\\
			$(ii)$	$w_{0,S\setminus \{\alpha_{8}\}}w_{0, S\setminus \{\alpha_{5}, \alpha_{8}\}}(\alpha_{8})=\omega_{8}-\alpha_{8}.$\\
			$(iii)$ ${v_{6}}^{-1}w_{0,S\setminus \{\alpha_{8}\}}w_{0, S\setminus \{\alpha_{5}, \alpha_{8}\}}(\alpha_{8})=2\alpha_{1}+3\alpha_{2}+4\alpha_{3}+6\alpha_{4}+5\alpha_{5}+3\alpha_{6}+2\alpha_{7}+\alpha_{8}.$
		\end{enumerate}
	\end{lem}
	\begin{proof}
		
		Proof of (1): (i) We consider the Dynkin subdiagram of $E_{8},$ corresponding to the subset $S\setminus \{\alpha_{1}\}$ of $S$ (see Figure: 4):

		
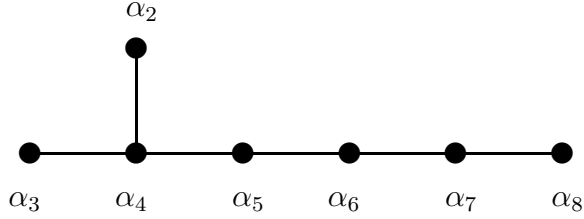
\vspace*{2cm}
		\begin{picture}(10,0)
			\thicklines
			\put(1.2,0){\circle*{0.2}}
			\put(1.2,0){\line(1,0){1}}
			\put(1,-0.5){$\alpha_{3}$}
			\put(2.2,0){\circle*{0.2}}
			\put(2,-0.5){$\alpha_{4}$}
			\put(2.2,0){\line(1,0){1}}
			\put(3.2,0){\circle*{0.2}}
			\put(3.1,-0.5){$\alpha_{5}$}
			\put(4.2,0){\circle*{0.2}}
			\put(4,-0.5){$\alpha_{6}$}
			\put(3.2,0){\line(1,0){1}}
			\put(2.2,1){\circle*{0.2}}
			\put(2.1,1.3){$\alpha_{2}$}
			\put(2.2,0){\line(0,1){1}}
			\put(5.2,0){\circle*{0.2}}
			\put(4.2,0){\line(1,0){1}}
			\put(5.1,-0.5){$\alpha_{7}$}
			\put(5.2,0){\line(1,0){1}}
			\put(6.2,0){\circle*{0.2}}
			\put(6.1,-0.5){$\alpha_{8}$}
		\end{picture}
		\vspace{1cm}
		
		Let $I=S\setminus\{\alpha_{1}\}.$
		Now we observe that $w_{0, S\setminus \{\alpha_{1}, \alpha_{8}\}}(\alpha_{8})=w_{0,I\setminus\{\alpha_{8}\}}(\alpha_{8}).$ We note that the connected Dynkin subdiagram of $E_{8}$ associated to $I$ containing $\alpha_{8}$ is of type $D_{7}.$  Since $\alpha_{8}$ is minuscule  in type $D_{7}$, by Lemma \ref{lem 2.3}, we have $w_{0,I\setminus\{\alpha_{8}\}}(\alpha_{8})=\alpha_{2}+\alpha_{3}+2\alpha_{4}+2\alpha_{5}+2\alpha_{6}+2\alpha_{7}+\alpha_{8}.$
		
		(ii). Follows from the observation that $w_{0, S\setminus \{\alpha_{8}\}}(\alpha_{i})=-\alpha_{i}$ for all $i\neq 8,$ and by using Lemma \ref{lemma 3.1}.
		
		(iii). By (ii) we have $w_{0,S\setminus\{\alpha_{8}\}}w_{0,S\setminus\{\alpha_{1},\alpha_{8}\}}(\alpha_{8})=\omega_{8}-(\alpha_{2}+\alpha_{3}+2\alpha_{4}+2\alpha_{5}+2\alpha_{6}+2\alpha_{7}+\alpha_{8}).$ Therefore, by using Lemma \ref{lemma 4.13}(3), we get ${v_{6}}^{-1}w_{S\setminus\{\alpha_{8}\}}w_{0, S\setminus\{\alpha_{1},\alpha_{8}\}}=\alpha_{7}+\alpha_{8}.$

		Proof of (2): (i) Next we consider the Dynkin subdiagram of $E_{8},$ corresponding to the subset $S\setminus \{\alpha_{2}\}$ of $S$:

		
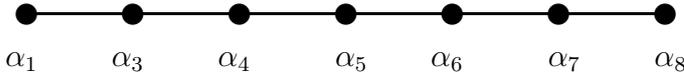
\vspace*{2cm}
		\begin{picture}(10,0)
			\thicklines
			\put(.2,0){\circle*{0.2}}
			\put(0,-0.5){$\alpha_{1}$}
			\put(0.2,0){\line(1,0){1}}
			\put(1.2,0){\circle*{0.2}}
			\put(1.2,0){\line(1,0){1}}
			\put(1,-0.5){$\alpha_{3}$}
			\put(2.2,0){\circle*{0.2}}
			\put(2,-0.5){$\alpha_{4}$}
			\put(2.2,0){\line(1,0){1}}
			\put(3.2,0){\circle*{0.2}}
			\put(3.1,-0.5){$\alpha_{5}$}
			\put(4.2,0){\circle*{0.2}}
			\put(4,-0.5){$\alpha_{6}$}
			\put(3.2,0){\line(1,0){1}}
			\put(5.2,0){\circle*{0.2}}
			\put(4.2,0){\line(1,0){1}}
			\put(5.1,-0.5){$\alpha_{7}$}
			\put(5.2,0){\line(1,0){1}}
			\put(6.2,0){\circle*{0.2}}
			\put(6.1,-0.5){$\alpha_{8}$}
		\end{picture}
		\vspace{1cm}
		
		Let $I=S\setminus\{\alpha_{2}\}.$
		Now we observe that $w_{0, S\setminus \{\alpha_{2}, \alpha_{8}\}}(\alpha_{8})=w_{0,I\setminus\{\alpha_{8}\}}(\alpha_{8}).$We note that the connected component of the Dynkin subdiagram of $E_{8}$ associated to $I$  containing $\alpha_{8}$ is of type $A_{7}.$  Since $\alpha_{8}$ is minuscule  in type $A_{7},$ by Lemma \ref{lem 2.3}, we have $w_{0,I\setminus\{\alpha_{8}\}}(\alpha_{8})=\alpha_{1}+\alpha_{3}+\alpha_{4}+\alpha_{5}+\alpha_{6}+\alpha_{7}+\alpha_{8}.$
		
		(ii). Follows from the observation that $w_{0, S\setminus \{\alpha_{8}\}}(\alpha_{i})=-\alpha_{i}$ for all $i\neq 8,$ and by using Lemma \ref{lemma 3.1}.
		
		(iii). By (ii) we have $w_{0,S\setminus\{\alpha_{8}\}}w_{0,S\setminus\{\alpha_{2},\alpha_{8}\}}(\alpha_{8})=\omega_{8}-(\alpha_{1}+\alpha_{3}+\alpha_{4}+\alpha_{5}+\alpha_{6}+\alpha_{7}+\alpha_{8}).$ Therefore, by using Lemma \ref{lemma 4.13}(3), we get ${v_{6}}^{-1}w_{S\setminus\{\alpha_{8}\}}w_{0, S\setminus\{\alpha_{2},\alpha_{8}\}}=\alpha_{1}+2\alpha_{2}+2\alpha_{3}+3\alpha_{4}+2\alpha_{5}+\alpha_{6}+\alpha_{7}.$
		
		Proof of (3)(i):
		Next we consider the Dynkin subdiagram of $E_{8},$ corresponding to the subset $S\setminus \{\alpha_{3}\}$ of $S:$ 
		
		
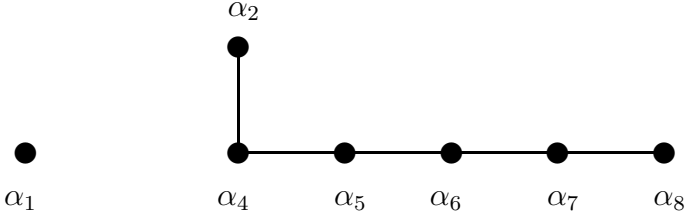
\vspace*{2cm}
		\begin{picture}(10,0)
			\thicklines
			\put(.2,0){\circle*{0.2}}
			\put(0,-0.5){$\alpha_{1}$}
			\put(2.2,0){\circle*{0.2}}
			\put(2,-0.5){$\alpha_{4}$}
			\put(2.2,0){\line(1,0){1}}
			\put(3.2,0){\circle*{0.2}}
			\put(3.1,-0.5){$\alpha_{5}$}
			\put(4.2,0){\circle*{0.2}}
			\put(4,-0.5){$\alpha_{6}$}
			\put(3.2,0){\line(1,0){1}}
			\put(2.2,1){\circle*{0.2}}
			\put(2.1,1.3){$\alpha_{2}$}
			\put(2.2,0){\line(0,1){1}}
			\put(5.2,0){\circle*{0.2}}
			\put(4.2,0){\line(1,0){1}}
			\put(5.1,-0.5){$\alpha_{7}$}
			\put(5.2,0){\line(1,0){1}}
			\put(6.2,0){\circle*{0.2}}
			\put(6.1,-0.5){$\alpha_{8}$}
		\end{picture}
		\vspace{1cm}
		
		Let $I=S\setminus \{\alpha_{3}\}.$ Then we observe that $w_{0,S\setminus \{\alpha_{3}, \alpha_{8}\}}(\alpha_{8})=w_{0,I\setminus\{\alpha_{8}\}}(\alpha_{8}).$Further, we note that the connected component of the Dynkin subdiagram associated to $I$ containing $\alpha_{8}$ is of type $A_{6}.$ Since $\alpha_{8}$ is minuscule  in type $A_{6},$ by Lemma \ref{lem 2.3}, we have $w_{0,I\setminus\{\alpha_{8}\}}(\alpha_{8})=\alpha_{2}+\alpha_{4}+\alpha_{5}+\alpha_{6}+\alpha_{7}+\alpha_{8}.$ 
		
		(ii). Follows from the observation that $w_{0, S\setminus \{\alpha_{8}\}}(\alpha_{i})=-\alpha_{i}$  for all $i\neq 8,$ and by using Lemma \ref{lemma 3.1}.
		
		(iii). By (ii) we have $w_{0,S\setminus\{\alpha_{8}\}}w_{0,S\setminus\{\alpha_{3},\alpha_{8}\}}(\alpha_{8})=\omega_{8}-(\alpha_{2}+\alpha_{4}+\alpha_{5}+\alpha_{6}+\alpha_{7}+\alpha_{8}).$ Therefore, by using Lemma \ref{lemma 4.13}(3), we get ${v_{6}}^{-1}w_{S\setminus\{\alpha_{8}\}}w_{0, S\setminus\{\alpha_{3},\alpha_{8}\}}=\alpha_{1}+\alpha_{2}+2\alpha_{3}+3\alpha_{4}+3\alpha_{5}+2\alpha_{6}+2\alpha_{7}+\alpha_{8}.$

		Proof of (4)(i): 
		Next we consider the Dynkin subdiagram of $E_{8},$ corresponding to the  subset $S\setminus\{\alpha_{4}\}$ of $S:$ 
		
		
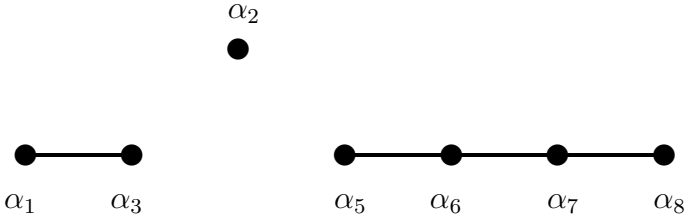
\vspace*{2cm}
		\begin{picture}(10,0)
			\thicklines
			\put(.2,0){\circle*{0.2}}
			\put(0,-0.5){$\alpha_{1}$}
			\put(0.2,0){\line(1,0){1}}
			\put(1.2,0){\circle*{0.2}}
			\put(1,-0.5){$\alpha_{3}$}
			\put(3.2,0){\circle*{0.2}}
			\put(3.1,-0.5){$\alpha_{5}$}
			\put(4.2,0){\circle*{0.2}}
			\put(4,-0.5){$\alpha_{6}$}
			\put(3.2,0){\line(1,0){1}}
			\put(2.2,1){\circle*{0.2}}
			\put(2.1,1.3){$\alpha_{2}$}
			\put(5.2,0){\circle*{0.2}}
			\put(4.2,0){\line(1,0){1}}
			\put(5.1,-0.5){$\alpha_{7}$}
			\put(5.2,0){\line(1,0){1}}
			\put(6.2,0){\circle*{0.2}}
			\put(6.1,-0.5){$\alpha_{8}$}
		\end{picture}
		\vspace{1cm}
		
		Let $I=S\setminus \{\alpha_{4}\}.$ Then we observe that $w_{0,S\setminus \{\alpha_{4}, \alpha_{8}\}}(\alpha_{8})=w_{0,I\setminus\{\alpha_{8}\}}(\alpha_{8}).$ Further, we note that the connected component of the Dynkin subdiagram associated to $I$ containing $\alpha_{8}$ is of type $A_{4}.$ Since $\alpha_{8}$ is minuscule  in type $A_{4}$, by Lemma \ref{lem 2.3}, we have $w_{0,I\setminus\{\alpha_{8}\}}(\alpha_{8})=\alpha_{5}+\alpha_{6}+\alpha_{7}+\alpha_{8}.$

		(ii). Follows from the observation that $w_{0, S\setminus \{\alpha_{8}\}}(\alpha_{i})=-\alpha_{i}$ for all $i\neq 8,$ and by using Lemma \ref{lemma 3.1}.
		
		(iii).  By 4(ii) we have $w_{0,S\setminus\{\alpha_{8}\}}w_{0,S\setminus\{\alpha_{4},\alpha_{8}\}}(\alpha_{8})=\omega_{8}-(\alpha_{5}+\alpha_{6}+\alpha_{7}+\alpha_{8}).$ Therefore, by using Lemma \ref{lemma 4.13}(3), we get ${v_{6}}^{-1}w_{S\setminus\{\alpha_{8}\}}w_{0, S\setminus\{\alpha_{4},\alpha_{8}\}}=\alpha_{1}+2\alpha_{2}+2\alpha_{3}+4\alpha_{4}+3\alpha_{5}+2\alpha_{6}+2\alpha_{7}+\alpha_{8}.$

		Proof of (5)(i): 
		Next we consider the Dynkin subdiagram of $E_{8},$ corresponding to the subset $S\setminus \{\alpha_{5}\}$ of $S:$

		
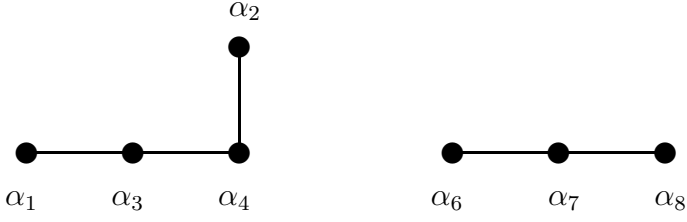
\vspace*{2cm}
		\begin{picture}(10,0)
			\thicklines
			\put(.2,0){\circle*{0.2}}
			\put(0,-0.5){$\alpha_{1}$}
			\put(0.2,0){\line(1,0){1}}
			\put(1.2,0){\circle*{0.2}}
			\put(1.2,0){\line(1,0){1}}
			\put(1,-0.5){$\alpha_{3}$}
			\put(2.2,0){\circle*{0.2}}
			\put(2,-0.5){$\alpha_{4}$}
			\put(4.2,0){\circle*{0.2}}
			\put(4,-0.5){$\alpha_{6}$}
			\put(2.2,1){\circle*{0.2}}
			\put(2.1,1.3){$\alpha_{2}$}
			\put(2.2,0){\line(0,1){1}}
			\put(5.2,0){\circle*{0.2}}
			\put(4.2,0){\line(1,0){1}}
			\put(5.1,-0.5){$\alpha_{7}$}
			\put(5.2,0){\line(1,0){1}}
			\put(6.2,0){\circle*{0.2}}
			\put(6.1,-0.5){$\alpha_{8}$}
		\end{picture}
		\vspace{1cm}
		
		Let $I=S\setminus \{\alpha_{5}\}.$ Then we observe that $w_{0,S\setminus \{\alpha_{5}, \alpha_{8}\}}(\alpha_{8})=w_{0,I\setminus\{\alpha_{8}\}}(\alpha_{8}).$Further, we note that the connected component of the Dynkin subdiagram associated to $I$ containing $\alpha_{8}$ is of type $A_{3}.$ Since $\alpha_{8}$ is minuscule  in type $A_{3},$ by Lemma \ref{lem 2.3}, we have $w_{0,I\setminus\{\alpha_{8}\}}(\alpha_{8})=\alpha_{6}+\alpha_{7}+\alpha_{8}.$ 
		
		(ii). Follows from the observation that $w_{0, S\setminus \{\alpha_{8}\}}(\alpha_{i})=-\alpha_{i}$  for all $i\neq 8,$ and by using Lemma \ref{lemma 3.1}.
		
		(iii).  By 5(ii) we have $w_{0,S\setminus\{\alpha_{8}\}}w_{0,S\setminus\{\alpha_{5},\alpha_{8}\}}(\alpha_{8})=\omega_{8}-(\alpha_{6}+\alpha_{7}+\alpha_{8}).$ Therefore, by using Lemma \ref{lemma 4.13}(3), we get ${v_{6}}^{-1}w_{S\setminus\{\alpha_{8}\}}w_{0, S\setminus\{\alpha_{5},\alpha_{8}\}}=\alpha_{1}+2\alpha_{2}+3\alpha_{3}+4\alpha_{4}+3\alpha_{5}+2\alpha_{6}+2\alpha_{7}+\alpha_{8}.$

		Proof of (6)(i): 
		Next we consider the Dynkin subdiagram of $E_{8},$ corresponding to  the subset $S\setminus \{\alpha_{6}\}$ of $S:$ 
		
		
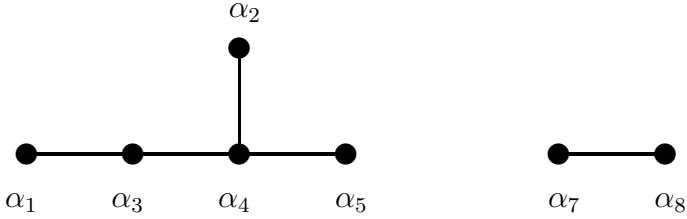
\vspace{2cm}
		\begin{picture}(10,0)
			\thicklines
			\put(.2,0){\circle*{0.2}}
			\put(0,-0.5){$\alpha_{1}$}
			\put(0.2,0){\line(1,0){1}}
			\put(1.2,0){\circle*{0.2}}
			\put(1.2,0){\line(1,0){1}}
			\put(1,-0.5){$\alpha_{3}$}
			\put(2.2,0){\circle*{0.2}}
			\put(2,-0.5){$\alpha_{4}$}
			\put(2.2,0){\line(1,0){1}}
			\put(3.2,0){\circle*{0.2}}
			\put(3.1,-0.5){$\alpha_{5}$}
			\put(2.2,1){\circle*{0.2}}
			\put(2.1,1.3){$\alpha_{2}$}
			\put(2.2,0){\line(0,1){1}}
			\put(5.2,0){\circle*{0.2}}
			\put(5.1,-0.5){$\alpha_{7}$}
			\put(5.2,0){\line(1,0){1}}
			\put(6.2,0){\circle*{0.2}}
			\put(6.1,-0.5){$\alpha_{8}$}
		\end{picture}
		\vspace{1cm}
		
		Let $I=S\setminus \{\alpha_{6}\}.$ Then we observe that $w_{0,S\setminus \{\alpha_{6}, \alpha_{8}\}}(\alpha_{8})=w_{0,I\setminus\{\alpha_{8}\}}(\alpha_{8}).$ Further, we note that the connected component of the Dynkin subdiagram associated to $I$ containing $\alpha_{8}$ is of type $A_{2}.$ Since $\alpha_{8}$ is minuscule  in type $A_{2},$ by Lemma \ref{lem 2.3}, we have $w_{0,I\setminus\{\alpha_{8}\}}(\alpha_{8})=\alpha_{7}+\alpha_{8}.$ 
		
		(ii). Follows from the observation that $w_{0, S\setminus \{\alpha_{8}\}}(\alpha_{i})=-\alpha_{i}$ for all $i\neq 8,$ and by using Lemma \ref{lemma 3.1}.
		
		(iii).   By 6(ii) we have $w_{0,S\setminus\{\alpha_{8}\}}w_{0,S\setminus\{\alpha_{6},\alpha_{8}\}}(\alpha_{8})=\omega_{8}-(\alpha_{7}+\alpha_{8}).$ Therefore, by using Lemma \ref{lemma 4.13}(3), we get ${v_{6}}^{-1}w_{S\setminus\{\alpha_{8}\}}w_{0, S\setminus\{\alpha_{6},\alpha_{8}\}}=2\alpha_{1}+2\alpha_{2}+3\alpha_{3}+4\alpha_{4}+3\alpha_{5}+2\alpha_{6}+2\alpha_{7}+\alpha_{8}.$

		Proof of (7)(i): Next we consider the Dynkin subdiagram of $E_{8},$ corresponding to the subset $S \setminus \{\alpha_{7}\}$ of $S:$

		
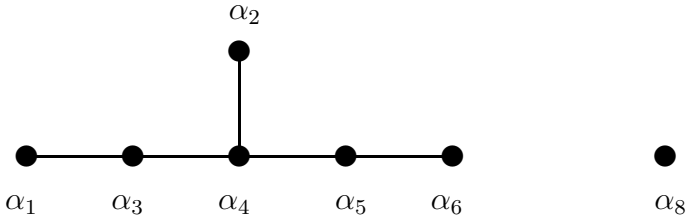
\vspace*{2cm}
		\begin{picture}(10,0)
			\thicklines
			\put(.2,0){\circle*{0.2}}
			\put(0,-0.5){$\alpha_{1}$}
			\put(0.2,0){\line(1,0){1}}
			\put(1.2,0){\circle*{0.2}}
			\put(1.2,0){\line(1,0){1}}
			\put(1,-0.5){$\alpha_{3}$}
			\put(2.2,0){\circle*{0.2}}
			\put(2,-0.5){$\alpha_{4}$}
			\put(2.2,0){\line(1,0){1}}
			\put(3.2,0){\circle*{0.2}}
			\put(3.1,-0.5){$\alpha_{5}$}
			\put(4.2,0){\circle*{0.2}}
			\put(4,-0.5){$\alpha_{6}$}
			\put(3.2,0){\line(1,0){1}}
			\put(2.2,1){\circle*{0.2}}
			\put(2.1,1.3){$\alpha_{2}$}
			\put(2.2,0){\line(0,1){1}}
			\put(6.2,0){\circle*{0.2}}
			\put(6.1,-0.5){$\alpha_{8}$}
		\end{picture}
		\vspace{1cm}
		
		Let $I=S\setminus \{\alpha_{7}\}.$ Then we observe that $w_{0,S\setminus \{\alpha_{7}, \alpha_{8}\}}(\alpha_{8})=w_{0,I\setminus\{\alpha_{8}\}}(\alpha_{8}).$ Further, we note that the connected component of the Dynkin subdiagram associated to $I$ containing $\alpha_{8}$ is of type $A_{1}.$ Since $\alpha_{8}$ is minuscule  in type $A_{1},$ by Lemma \ref{lem 2.3}, we have $w_{0,I\setminus\{\alpha_{8}\}}(\alpha_{8})=\alpha_{8}.$ 
		
		(ii). Follows from the observation that $w_{0, S\setminus \{\alpha_{8}\}}(\alpha_{i})=-\alpha_{i}$ for all $i\neq 8,$ and by using Lemma \ref{lemma 3.1}.

		(iii). By 7(ii) we have $w_{0,S\setminus\{\alpha_{8}\}}w_{0,S\setminus\{\alpha_{7},\alpha_{8}\}}(\alpha_{8})=\omega_{8}-\alpha_{8}.$ Therefore, by using Lemma \ref{lemma 4.13}(3), we get ${v_{6}}^{-1}w_{S\setminus\{\alpha_{8}\}}w_{0, S\setminus\{\alpha_{7},\alpha_{8}\}}=2\alpha_{1}+3\alpha_{2}+4\alpha_{3}+6\alpha_{4}+5\alpha_{5}+3\alpha_{6}+2\alpha_{7}+\alpha_{8}.$

	\end{proof}
	\begin{lem}\label{lemma 5.15}Let $w_{i}=w_{0,S\setminus\{\alpha_{i}, \alpha_{8}\}}w_{0, S\setminus\{\alpha_{8}\}}v_{6}$ for $1\le i\le 8.$  Then we have $w_{i}^{-1}(\alpha_{j})$ is a positive root for $j\neq i,$ and $w_{i}^{-1}(\alpha_{i})$ is negative root.
	\end{lem}
	\begin{proof}
		Note that for $i=8$ we have $w_{8}=v_{6}.$ Then by Lemma \ref{lemma 4.13}, we are done. For $i\neq 8,$ let  $w_{0,S\setminus \{\alpha_{i}, \alpha_{8}\}}(\alpha_{i})=\beta.$ Then $\beta$ is positive root whose support does not contain $\alpha_{8}.$ Since $w_{0, S\setminus \{\alpha_{8}\}}$ is the longest element of $W_{S\setminus \{\alpha_{8}\}},$  $w_{0,S\setminus \{\alpha_{8}\}}(\alpha_{i})=-\alpha_{i}$ for $i\neq 8.$ Thus we have $w_{0,S\setminus \{\alpha_{8}\}}(\beta)=-\beta.$ Further, since the support of $\beta$ does not contain $\alpha_{8},$ by Lemma \ref{lemma 5.8}(3) we have $v_{6}^{-1}(-\beta)$ is a negative root. Hence we have $w_{i}^{-1}(\alpha_{i})$ is a negative root for $i\neq 8.$ 
		
		On the other hand, for $i\neq8,$ and $j\neq 8,i,$  $w_{0,S\setminus \{\alpha_{i}, \alpha_{8}\}}(\alpha_{j})=-\alpha_{k},$ where $\alpha_{k}$ is a simple root different from $\alpha_{8}.$ Therefore, $w_{0,S\setminus \{\alpha_{8}\}}(-\alpha_{k})=\alpha_{k}.$ Further, since $\alpha_{k}$ is different from $\alpha_{8},$ by Lemma \ref{lemma 4.13}(3), $v_{6}^{-1}(\alpha_{k})$ is a positive root. Hence $w_{i}^{-1}(\alpha_{j})$ is a positive root when $i\neq 8,$ and $j\neq i, 8.$ Moreover, by Lemma \ref{lemma 5.14}, we have  $w_{i}^{-1}(\alpha_{8})$ is a positive root for $i\neq 8.$ Hence, we have $w_{i}^{-1}(\alpha_{j})$ is a positive root for $i\neq 8,$ and $j\neq i.$
	\end{proof}
	
	\begin{proof}[{\bf Proof of Proposition \ref{proposition 5.12}:}]
		Recall that $w_{i}=w_{0,S\setminus \{\alpha_{i}, \alpha_{8}\}} w_{0, S\setminus \{\alpha_{8}\}} v_{6}.$ Note that $w_{i}(\alpha_{7})=w_{S\setminus \{\alpha_{i}, \alpha_{8}\}} w_{0, S\setminus \{\alpha_{8}\}} v_{6}(\alpha_{7}).$ By Lemma \ref{lemma 4.13}(2), we conclude  that $w_{i}(\alpha_{7})$ is a non simple positive root.
		On the other hand, by Lemma \ref{lemma 5.15},  $w_{i}^{-1}(\alpha_{i})$ is a negative root and $w_{i}^{-1}(\alpha_{j})$ is a positive root for $j\neq i.$ Thus $P_{i}$ is the stabilizer of $X_{P_{7}}(w_{i})$ in $G.$ Since $\alpha_{0}=\omega_{8},$ and  $v_{6}^{-1}(\alpha_{0})$ is a negative root, $w_{i}^{-1}(\alpha_{0})=v_{6}^{-1}(\alpha_{0})$ (as$~w_{0,S\setminus\{\alpha_{8}\}}w_{0,S\setminus\{\alpha_{i}, \alpha_{8}\}}(\alpha_{0})=\alpha_{0}$) is a negative root. Therefore, by using Theorem \ref{thm1} the natural homomorphism $P_{i}\longrightarrow Aut^0(X_{P_{7}}(w_{i}))$ is an isomorphism of algebraic groups. 
	\end{proof}

	\section{main theorem}
	In this section, our goal is to prove the following. A fundamental weight $\omega_{i}$ is minuscule if and only if for any parabolic subgroup $Q$ containing $B$ properly, there is no Schubert variety $X_{Q}(w)$ in $G/Q$ such that $P_{i}=Aut^{0}(X_{Q}(w)).$

	\begin{thm}\label{theorem 4.1}
		Assume that $\omega_{r}$ is  minuscule. If there exists a Schubert variety $X_{Q}(w)$ in  $G/Q$ (for some parabolic subgroup $Q$ of $G$ containing $B$) such that $P_{r}=Aut^0(X_{Q}(w)),$ then we have $Q=B.$
	\end{thm}
	\begin{proof} Assume that $P_{r}=Aut^0(X_{Q}(w))$ where $X_{Q}(w)$ is a Schubert variety in $G/Q,$ for some $Q\neq B.$ Then there exists a non-empty subset $J$ of $S$ such that $Q=P_{J}.$ Since $P_{r}=Aut^0(X_{Q}(w)),$ we have $w^{-1}(\alpha_{j})>0$ for all $j\neq r,$ and we have $w^{-1}\in W^{S \setminus \{\alpha_{r}\}}.$ Since the action $P_{r}$ on $X_{Q}(w)$ is faithful, by Theorem \ref{thm1} we have $w^{-1}(\alpha_{0})<0.$ Further, since $w^{-1}(\alpha_{0})<0$ and $w^{-1}\in W^{S \setminus \{\alpha_{r}\}},$ by Lemma \ref{lem 2.4} we have $w^{-1}=w_{0}^{S \setminus \{\alpha_{r}\}}.$ Thus we have $w=({w_{0}^{S \setminus \{\alpha_{r}\}}})^{-1}.$ Note that $(w_{0}^{S\setminus\{\alpha_{r}\}})^{-1}=w_{0,S\setminus\{\alpha_{r}\}}w_{0}=w_{0}(w_{0}w_{0,S\setminus\{\alpha_{r}\}}w_{0})=w_{0}w_{0,S\setminus\{\alpha_{\sigma(r)}\}}.$ Therefore, we have $w={w_{0}^{S \setminus \{\alpha_{\sigma(r)}\}}}.$ Hence, we have $\sigma(J)\subset S\setminus \{\alpha_{r}\}.$ Let $\tau$ be the Dynkin diagram automorphism of $S\setminus \{\alpha_{r}\}$ induced by  $-w_{0, S \setminus \{\alpha_{r}\}}.$ Now since $J$ is nonempty, there exists $\alpha_{i} \in J$ such that $\alpha_{\sigma(i)}\neq \alpha_{r}.$ So there exists $\alpha_{j}\in S\setminus \{\alpha_{r}\}$ such that $\alpha_{\tau(j)}=\alpha_{\sigma(i)}.$ That is we have $\alpha_{\sigma\tau(j)}=\alpha_{i}.$ Now we have $w^{-1}s_{j}w=w_{0}w_{0, S\setminus \{\alpha_{r}\}}s_{j}w_{0, S\setminus \{\alpha_{r}\}}w_{0}=s_{\sigma\tau(j)}=s_{i}.$ This is a contradiction to the fact that $P_{r}$ is the stabilizer of $X_{Q}(w)$ in $G.$ Hence, we have $Q=B.$
	\end{proof}
	
	\begin{thm}\label{theorem 10.1}
		A fundamental weight $\omega_{\alpha}$ is  minuscule if and only if for every parabolic subgroup $Q$ containing $B$ properly, there is no Schubert variety $X_{Q}(w)$ in $G/Q$ such that $P_{\alpha}=Aut^0(X_{Q}(w)).$
	\end{thm}
	\begin{proof}
		Follows from Proposition \ref{Prop 6.1},  Proposition \ref{proposition 5.4}, Proposition \ref{Proposition 5.7}, Proposition \ref{proposition 5.12}, and  Theorem \ref{theorem 4.1}.
	\end{proof}
	
In the following remark we label the simple roots for $B_{2}$ as in \cite[p.58]{Hum1}, namely $\langle \alpha_{1},\alpha_{2} \rangle=-2,$ and $\langle \alpha_{2},\alpha_{1}\rangle=-1.$ 
	\begin{rem}\label{rmk1}
		In non-simply-laced type, Theorem \ref{theorem 10.1} may not hold.
		For instance consider $G=SO(5,\mathbb{C})$ and $w=s_{2}s_{1}.$ Here, $\omega_{2}$ is the unique minuscule weight. 	Then, $X_{P_{2}}(w)$ is a Schubert variety in $G/P_{2}$ such that $P_{2}=Aut^0(X_{P_{2}}(w)).$ 
	\end{rem}
	
\begin{proof}
	
Note that the tangent bundle of $G/P_2$ is $\mathcal{L}(\mathfrak{g/p_{2}}).$ Further, $\mathfrak{g/p_{2}}$ has a filtration of $B$-submodules with successive quotients are of the form $\mathbb{C}_{\beta}$ for some $\beta\in \{\alpha_{1},\alpha_{1}+\alpha_{2}, \alpha_{1}+2\alpha_{2}\}.$ Then we use Lemma \ref{lem2.3}, Lemma \ref{lem2.4}, and Lemma \ref{lem 2.1}, to compute the following cohomology modules: $H^0(s_{1}, \mathfrak{g/p_{2}})=\mathbb{C}_{\alpha_{1}}\oplus \mathbb{C}h(\alpha_{1})\oplus \mathbb{C}_{-\alpha_{1}}\oplus \mathbb{C}_{\alpha_{1}+\alpha_{2}}\oplus \mathbb{C}_{\alpha_{2}}\oplus \mathbb{C}_{\alpha_{1}+2\alpha_{2}}$ and $H^0(w,\mathfrak{g/p_{2}})=H^0(s_{2}s_{1},\mathfrak{g/p_{2}})=H^0(s_{2}, H^0(s_{1}, \mathfrak{g/p_{2}}))=\mathbb{C}_{\alpha_{1}}\oplus \mathbb{C}_{\alpha_{1}+\alpha_{2}}\oplus \mathbb{C}_{\alpha_{1}+2\alpha_{2}}\oplus \mathbb{C}h(\alpha_{1})\oplus \mathbb{C}_{-\alpha_{1}}\oplus \mathbb{C}_{-\alpha_{1}-\alpha_{2}}\oplus \mathbb{C}_{-\alpha_{1}-2\alpha_{2}}\oplus \mathbb{C}_{\alpha_{2}}\oplus \mathbb{C}h(\alpha_{2})\oplus \mathbb{C}_{-\alpha_{2}}\simeq \mathfrak{g}$ as  $T$-module. Note that $\mathbb{C}_{\alpha_{1}}$ is a $B$-submodule of $\mathfrak{g/p_{2}}.$ Thus, $H^0(w,\alpha_{1})$ is a $B$-submodule of $H^0(w,\mathfrak{g/p_{2}}).$

But, $H^0(w,\alpha_{1})_{-\alpha_{0}}\neq 0,$ as $s_{2}s_{1}(\alpha_{1})=-\alpha_{0}.$ Hence, $H^0(w, \mathfrak{g/p_{2}})_{-\alpha_{0}}\neq 0.$ Therefore, the restriction map $\varphi: H^0(G/P_2, \mathcal{L}(\mathfrak{g/p_{2}}))\longrightarrow H^0(w,\mathfrak{g/p_{2}})$ is injective. Furthermore, since $G$ acts on $G/P_2,$ we have $\mathfrak{g}\subseteq H^0(G/P_2,\mathcal{L}(\mathfrak{g/p_{2}})).$ Thus by the above discussion, $\varphi$ is an isomorphism and hence we have $H^0(G/P_2,\mathcal{L}(\mathfrak{g/p_{2}}))=\mathfrak{g}$ and $H^0(w,\mathfrak{g/p_{2}})=\mathfrak{g}$ as $B$-modules. On the other hand, $T_{X_{P_2}(w)}$ is a subsheaf of the restriction of the tangent bundle $T_{G/P_2}$ to $X_{P_2}(w).$ Hence, $H^0(X_{P_2}(w), T_{X_{P_2}(w)})$ is a $B$-submodule of $H^0(w,\mathfrak{g/p_{2}}).$ Further, since $H^0(w,\mathfrak{g/p_{2}})=\mathfrak{g},$ Lie$(Aut^0(X_{P_2}(w)))$ is a Lie subalgebra of $\mathfrak{g}$ containing $\mathfrak{b}.$ Thus $Aut^0(X_{P_2}(w))$ is a parabolic subgroup of $G$ containing $B.$ 
	
	Let $P'=Aut^0(X_{P_2}(w))\subseteq G.$ Let $\alpha$ be a simple root such that $s_{\alpha}\in W(P'),$ Weyl group of $P'.$ Let $n_{\alpha}\in N_{P'}(T)$ be a representative of $s_{\alpha}.$ Then $n_{\alpha}\cdot wP_{2}/P_{2}$ is a $T$-fixed point in $X_{P_2}(w).$ Further, for any dominant character $\chi$ of $P_{2},$ $T$ acts on the fiber of the $\mathcal{L}(\chi)$ over the point $n_{\alpha}\cdot wP_{2}/P_{2}$ by the character $s_{\alpha}w(\chi).$ Therefore, $n_{\alpha}\cdot wP_{2}/P_{2}=s_{\alpha}wP_{2}/P_{2}.$ Since $s_{\alpha}wP_{2}/P_{2}\in X_{P_{2}}(w),$ we have $s_{\alpha}w\in W^{\{\alpha_{2}\}}$ and $s_{\alpha}w< w$ or $s_{\alpha}wP_{2}=wP_{2}.$ Thus, $n_{\alpha} \text{~is in the stabilizer of ~} X_{P_2}(w)$ in $G.$ Since $P'$ is generated by $B$ and $n_{\alpha},$ where $\alpha\in S$ such that $s_{\alpha}\in W(P'),$ it follows that $P'\text{~is the stabilizer of~} X_{P_2}(w)$ in $G.$ Hence, $Aut^0(X_{P_2}(w))$ is equal to the stabilizer of $X_{P_2}(w)$ in $G.$ Now, the remark follows from the fact that stabilizer of $X_{P_2}(w)$ in $G$ is $P_2.$ 
	
We now describe $X_{P_2}(w)$ geometrically. Consider the Schubert variety $X(w)$ in $G/B$. Consider the open subsets
$BwB/B\subseteq X(w)$ and $BwP_{2}/P_2\subseteq X_{P_2}(w).$ 

Let $U=\prod\limits_{\alpha\in R^{+}(w^{-1})}U_{-\alpha}.$ The morphisms $U\longrightarrow BwB/B$ sending $u\mapsto uwB/B$ and $U\longrightarrow BwP_{2}/P_{2}$ sending $u\mapsto uwP_{2}/P_{2}$ are isomorphisms (see \cite[Section 2.1, (2) and (3), p.62]{BK}). Therefore, the class groups of $BwB/B$ and $BwP_{2}/P_{2}$ are trivial.

Further, $X_{P_2}(w)\setminus BwP_2/P_2=X_{P_2}(s_{1})$ is an irreducible divisor. Also, $X(w)\setminus BwB/B$ is union of two irreducible divisors $X(s_{1})$ and $X(s_{2}).$ By \cite[Proposition 6.5(c), p.133]{Har}, $rank(Cl(X(w))\le 2$ and $rank(Cl(X_{P_2}(w)))\le 1,$ where $Cl(X(w))$ (respectively, $Cl(X_{P_2}(w))$) denotes the class group of $X(w)$ (respectively, of $X_{P_2}(w)$). 

Hence, $rank(Pic(X(w))\le 2$ and $rank(Pic(X_{P_2}(w)))\le 1,$ where $Pic(X(w))$ (respectively, $Pic(X_{P_2}(w))$) denotes the Picard group of $X(w)$ (respectively, of $X_{P_2}(w)$). Further, the restriction of the line bundles $\mathcal{L}(\omega_{1})$ and $\mathcal{L}(\omega_{2})$ to $X(w)$ are linearly independent. Therefore, $rank(Pic(X(w)))=2$ and $rank(Pic(X_{P_2}(w)))=1.$  
 
Now, consider the restriction of the morphism $\pi: G/B\longrightarrow G/P_2$ to $X(w),$ which also we denote it by $\pi.$ So, $\pi: X(w)\longrightarrow X_{P_2}(w)$ is a surjective birational morphism such that $\pi_{*}\mathcal{O}_{X(w)}=\mathcal{O}_{X_{P_2}(w)}$ (see \cite[Theorem 3.3.4(a), p.96]{BK}). Therefore, by \cite[Corollary 2.2, p.45]{Bri2}, $\pi$ induces homomorphism $\pi_{*}:Aut^0(X(w))\longrightarrow Aut^0(X_{P_2}(w))$ of algebraic groups. On the other hand, since $w^{-1}(\alpha_{0})<0,$ by \cite[Theorem 6.6, p.781]{Kan} $P_{2}\subseteq Aut^0(X(w)).$  Therefore, $\pi_{*}: Aut^0(X(w))\longrightarrow Aut^0(X_{P_2}(w))=P_2$ is an isomorphism.
	
\end{proof}

	{\bf Acknowledgements} We thank the referee for useful comments and suggestions. The first named author thanks the Infosys Foundation for the partial financial support. He also thanks MATRICS for the partial financial support. The second named author thanks the Chennai Mathematical Institute for the hospitality during his stay.

\end{document}